\documentclass[pdflatex]{article} 
\pdfoutput=1
\usepackage{amsmath,times}
\usepackage{amsthm}
\usepackage{amssymb}
\usepackage{latexsym}
\usepackage{amscd}
\usepackage[english]{babel}
\usepackage[latin1]{inputenc}
\usepackage{graphicx}
\usepackage{verbatim}
\newtheorem{theorem}{Theorem}[section]
\newtheorem{lemma}[theorem]{Lemma}
\newtheorem{proposition}[theorem]{Proposition}

\theoremstyle{definition}
\newtheorem{definition}[theorem]{Definition}

\newtheorem{remark}[theorem]{Remark}

\begin{document} 

\title{The Hele-Shaw Flow and Moduli of Holomorphic Discs}
\author{Julius Ross and David Witt Nystr\"om}
\date{12 January 2014}
\maketitle

\pagestyle{headings}

\abstract{We present a new connection between the Hele-Shaw flow, also known as two-dimensional (2D) Laplacian growth, and the theory of holomorphic discs with boundary contained in a totally real submanifold.  Using this we prove short time existence and uniqueness of the Hele-Shaw flow with varying permeability both when starting from a single point and also starting from a smooth Jordan domain.  Applying the same ideas we prove that the moduli space of smooth quadrature domains is a smooth manifold whose dimension we also calculate, and we give a local existence theorem for the inverse potential problem in the plane.}

\section{Introduction}

The Hele-Shaw flow is a model for describing the propagation of fluid in a Hele-Shaw cell.   Such a cell consists of two parallel plates separated by a small gap, and the fluid is confined to the narrow space between them.   This moving boundary model has been intensely studied for over a century, and is a paradigm for understanding more complicated systems such as the flow of water in a porous media, the melting of ice and models of tumour growth.  


\begin{figure}[htb]
\includegraphics[width=0.8\textwidth, trim=50 130 230 130]{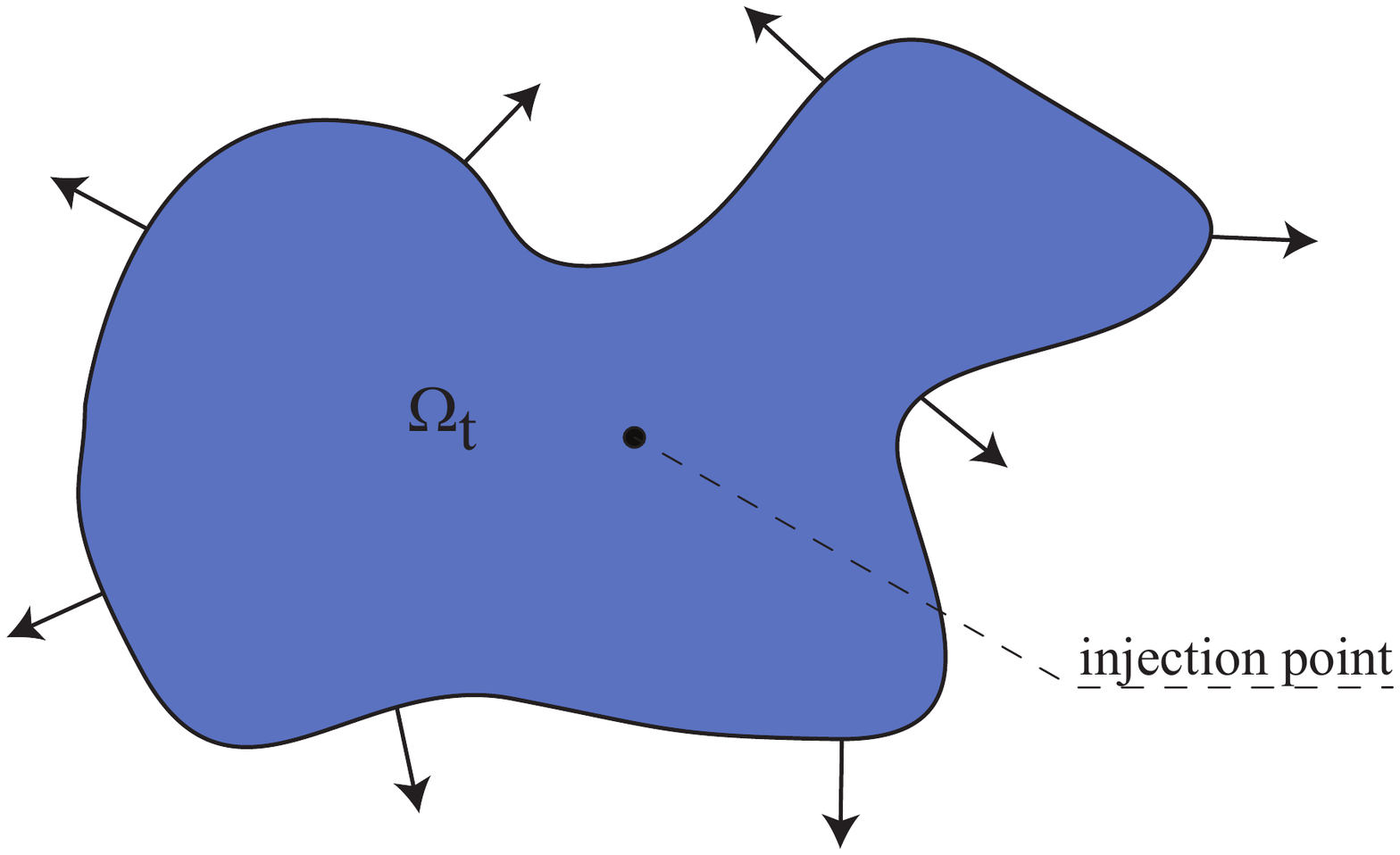}
\centering
\caption{Dynamics of the Hele-Shaw flow}  
\end{figure}

This flow is essentially two-dimensional, and one identifies the region occupied by fluid with a subset $\Omega_t$ of the plane $\mathbb{R}^2.$  In this paper we will consider the case of  a viscous incompressible fluid that is surrounded by air while new fluid is injected at a constant rate at the origin. If $p$ is the pressure in the fluid then the velocity $V$ of the fluid at that point (or really the mean velocity along the gap between the plates) is calculated as  
\begin{equation} \label{classic1}
V=-\nabla p.
\end{equation} 
Here we have neglected some physical constants. The pressure in the fluid is harmonic except at the origin, where it will have a logarithmic singularity, and is zero on the boundary since we assume that both the air pressure and surface tension is zero. 

To model this, let $\{\Omega_t\}$ be a family of domains in $\mathbb R^2$ containing the origin, and let $p_t$ denote the function which is zero on $\partial \Omega_t$ and which solves 
\[\Delta p_t=-\delta_0\]
 on $\Omega_t,$ where $\delta_0$ is the Dirac measure at the origin.   Then $\Omega_t$  is called a solution to the (classical) Hele-Shaw flow if for all $t$
\begin{equation} \label{classic2}
V_t=-\nabla p_t \quad \text{ on } \partial \Omega_t,
\end{equation} 
where $V_t$ is the normal outward velocity of $\partial\Omega_t$.  Clearly we need some regularity of the family $\{\Omega_t\}$ in order for this to make sense, thus the question of local existence of a classical solution to the Hele-Shaw flow with a given initial domain $\Omega_0$ is non-trivial.

Similarly one can model the case of a Hele-Shaw cell where the fluid moves through a porous material.  If the permeability of this material is given by a positive function $\kappa$ then,  by Darcy's law, the equation of motion of a fluid in the cell becomes
\begin{equation} \label{classic3}
V_t=-\kappa \nabla p_t \quad \text{ on } \partial \Omega_t,
\end{equation}
with $p_t$ defined as before.  This is a special case of elliptic growth of Beltrami type considered in \cite{Khavinson}. It is also equivalent to studying the classical Hele-Shaw flow on a Riemann surface, with the Riemannian metric encoding the permeability \cite{Hedenmalm}.\medskip

There is a vast literature on the Hele Shaw flow (see \cite{Gustafsson} and the references therein). The main short-time existence and uniqueness result for classical solutions to the Hele-Shaw flow in the case $\kappa\equiv 1$ is due to Kufarev and Vinogradov \cite{Vinogradov} and dates from 1948. It states that if $\Omega_0$ is a simply connected domain with real analytic boundary, then there exists a unique (indeed real analytic) solution $\{\Omega_t\}_{t\in(-\epsilon,\epsilon)}$ to the Hele-Shaw flow for some $\epsilon>0$.  Observe here that the solution extends both forward and backward in time.  In \cite{Reissig} Reissig and von Wolfersdorf gave a new proof of this result using a non-linear version of the Cauchy-Kovalevskaya theorem due to Nishida. Tian \cite{Tian} provided yet another proof relying on properties of the Cauchy integral of the free boundary, and recently Lin \cite{Lin} proved the same result using a result of Gustafsson on rational  solutions \cite{Gustafsson3} combined with a new perturbation theorem to get to the the general case.

There is also a slightly different setting of the Hele-Shaw flow, where the source of fluid is not a point but a curve inside the starting domain (see e.g. \cite{Antontsev, Elliott, Escher2}). In \cite{Escher} Escher and Simonett proved short-time existence of classical solutions to the Hele-Shaw flow in this setting, and in \cite{Escher2,Escher3} they proved short-time existence of classical solutions to the Hele-Shaw flow with surface tension. There is also related work of  Hanzawa \cite{Hanzawa} on classical solutions to the Stefan problem.

Less has been written about the case of varying permeability.   In \cite{Hedenmalm} Heden\-malm-Shimorin proved short-time existence under the assumption that $\kappa$ is real analytic, and the starting domain has real analytic boundary. They also proved short-time existence when the starting domain is empty, under an additional assumption that $\Delta \log \kappa >0,$ which translates to saying that the Riemann surface in question has negative curvature. See also \cite{Hedenmalm2} for related results.

\section{Summary of Results}

Our first main result concerns the existence of the Hele-Shaw flow with empty initial condition.

\begin{theorem}
Let $\kappa$ be a positive smooth function.  Then there is an $\epsilon>0$ and family $\Omega_t$  of domains in $\mathbb C$ containing the origin that satisfies the Hele-Shaw flow with permeability $\kappa$ and such that$$\operatorname{Area}(\Omega_t) =t \quad \text{ for } 0<t<\epsilon$$
(i.e. the limit of $\Omega_t$ as $t$ tends to zero is just the origin).
\end{theorem}

We remark that this improves on the  results of Hedenmalm-Shimorin since we only assume that $\kappa$ is smooth.     Despite this having a rather natural physical interpretation, there appear to be few other results with empty initial condition.  This may be because many of the techniques in the study of the Hele-Shaw flow rely on identifying the starting domain with the unit disc (say through the Riemann Mapping Theorem) and thus cannot be applied.  \medskip


Our second result on the Hele-Shaw flow is short term existence and uniqueness of from a non-empty initial condition.

\begin{theorem}\label{thm:heleshawmain}
Let $\Omega_0$ be smooth Jordan domain containing the origin and let $\kappa$ be a smooth positive function defined in a neighbourhood of $\partial \Omega_0.$   Then there exists an $\epsilon >0$ such that there exists a unique smooth increasing family of domains $\{\Omega_t\}_{t\in [0,\epsilon)}$ which solves the Hele-Shaw flow with permeability $\kappa$ forward in time.
\end{theorem}

We remark that this goes beyond the results of Hedenmalm-Shimorin since we do not require $\kappa$ or the boundary of the initial domain to be real analytic.  However if we do make these assumptions then we get an alternative proof of their result that holds both forwards and backwards in time. \medskip

Our approach will involve controlling the complex moments of the flow $\Omega_t$, which fits into the more general framework of quadrature domains (as advocated by Gustafsson and Shapiro).  We say that a bounded domain $\Omega$ is a \emph{quadrature domain} if there are points $z_1,\ldots,z_m$ in $\Omega$ and complex coefficients $c_{kj}$ for $1\le k\le m$ and $0\le j\le n_{k}-1$ such that the quadrature identity
$$\int_\Omega f dA = \sum_{k=1}^m \sum_{j=1}^{n_k-1} c_{kj} f^{(j)}(z_k)$$
holds for all integrable holomorphic functions $f$ in $\Omega$.  Then the integer $n=\sum_{k=1}^m n_k$ is called the order of the domain.  

Using the same ideas that we use to study the Hele-Shaw flow we shall prove the following about such objects.

\begin{theorem} \label{quadraturetheorem}
The moduli space of smooth quadrature domains of order $n$ and connectivity $c$ is a smooth manifold of real dimension $4n+c-2.$
\end{theorem}

The fact that the moduli space of smooth quadrature domains is a manifold appears to be previously unknown.  At a generic point its dimension can be seen in the following way.  Suppose that the points $z_k$ are all distinct and $n_{k}=1$ for all $k$.  Then the quadrature identity becomes
\begin{equation}
\int_{\Omega} f dA = \sum_{k=1}^m c_k f(z_k)\label{eq:quadsimple}
\end{equation}
for some $c_k\in \mathbb C$, and points $z_1,\ldots,z_m$ where now $n=m$ is the order of $\Omega$.  It was proved by Gustafsson that a generic smooth $\Omega$ of connectivity $c$ satisfying \eqref{eq:quadsimple} moves in at least a $c-1$ dimensional family that still satisfies \eqref{eq:quadsimple} (see \cite[p11]{Gustafsson6} and \cite[Thm12]{Gustafsson4}).  Then we have a further $2n$-real parameters available by moving the points $z_k$, and a further $2n$-real parameters available by moving the complex coefficients $c_k$, subject to the constraint that $c_1+\cdots+c_n\in \mathbb R$, coming from putting in $f=1$ into \eqref{eq:quadsimple} and noting that the area of $\Omega$ must be real-valued.  Thus in total we have $4n+c-2$ real parameters, and the above theorem implies that there are no more, as well as proving this dimension at a non-generic point. 

\begin{center}
  *
\end{center}

In our opinion, the interest of this work lies not only in the particulars of the above theorems but also in the techniques introduced for their proof.   We shall show that a solution to the Hele-Shaw flow is equivalent to being able to lift $\Omega_t\subset \mathbb R^2=\mathbb C$ to a family of holomorphic discs $\Sigma_t$ in $\mathbb C\times \mathbb P^1$ such that 

\begin{enumerate}
\item $\Omega_t$ is the image of the projection of $\Sigma_t$ to $\mathbb{C}$ and
\item The discs $\Sigma_t$ attach along their boundaries to a certain totally real submanifold $\Lambda$ of  $\mathbb{C}^2\subset \mathbb{C}\times \mathbb{P}^1.$ and
\item The closure of $\Sigma_t$ intersects $\mathbb{C}\times \{\infty\}\subset \mathbb C\times \mathbb P^1$ only at the point $(0,\infty).$ (See Figure 2).
\end{enumerate}


\begin{figure}[htb]
  \centering
\includegraphics[width=0.8\textwidth, trim=0 50  100 70]{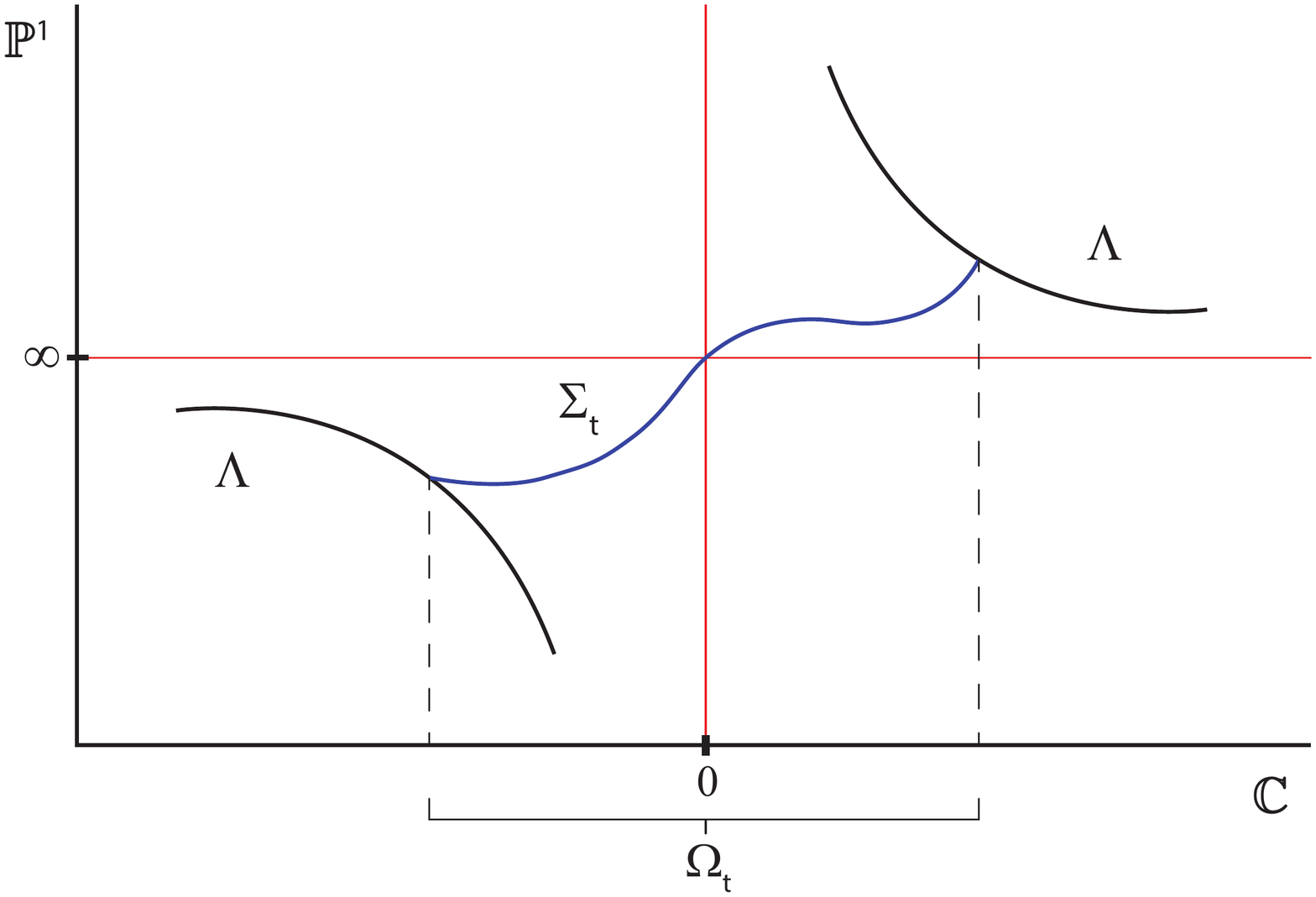} 
\caption{Lifting the Hele-Shaw flow}
\end{figure}

This correspondence comes about through thinking of the lifts $\Sigma_t$ as the graph of certain ``Schwarz'' functions on $\Omega_t$ that are holomorphic except for a simple pole at the origin.   Starting with $t=0$, we construct a smooth strictly subharmonic function $\phi$ (that encodes the permeability by $\kappa = 1/\Delta \phi$) such that $\Omega_0$ admits such a Schwarz function $S_0$.   Then setting
$$ \Lambda = \{ (z, \frac{\partial \phi}{\partial z}) :\} \subset \mathbb C^2,$$
the graph of $S_0$ will be  a  holomorphic disc $\Sigma_0$ whose boundary lies in $\Lambda$.

Once this is done we are in a position to apply the well-developed theory of embedded holomorphic discs.  For the proof of the existence of the Hele Shaw flow with empty initial condition we will rely on a connection between holomorphic discs and solutions to the Homogeneous Monge Amp\'ere equation, ultimately relying on an openness theorem of Donaldson.  For the other short time existence theorem we will consider the moduli space of embedded discs nearby $\Sigma_0$ that are attached to $\Lambda$ and that pass through $(0,\infty)$.  We shall show this moduli space is a smooth manifold of real dimension one (which immediately gives both the existence and uniqueness), which relies on a standard argument using the associated Schottky double.\medskip

This lifting interpretation of the Hele-Shaw flow through Schwarz functions comes about through looking at the complex moments
$$ M_k(t) = \int_{\Omega_t} z^k \frac{1}{\kappa} dA$$
where $dA$ denotes the Lebesgue measure on $\mathbb C$.  It is known that the Hele-Shaw flow is characterised by the condition that $M_k(t)$ is constant with $t$ for $k\ge 1$ (and so the area $M_0(t)$ is linear with $t$).    One can similarly study the problem of flows $\Omega_t$ for which  $M_k(t)$ for $k\ge 1$ are allowed to vary in a particular way, and this puts the Hele-Shaw flow into the more general picture of ``inverse potential problems''.  We use the same approach as above to prove a statement about these more general flows:

\begin{theorem}\label{thm:summarymoment}
Let $\Omega$ be a smooth Jordan domain, and $V$ be a smooth, nowhere vanishing outward pointing normal vector field on $\partial \Omega$.   Then there exists a smooth variation $\Omega_t$ for $t\in [0,\epsilon)$ for some $\epsilon>0$ such that $$\frac{d}{dt}\big\vert_{t=0}\partial \Omega_t=V$$ and whose moments vary linearly in $t$.  
\end{theorem}

Note that this theorem is not immediately obvious since a-priori there could be relations among the higher moments; thus we think of it as showing a kind of independence among them.   In fact it is known that if $\kappa$ is real analytic, and $\Omega$ has real analytic boundary, then the moments $M_k$ for $k\ge 0$ provide parameters for the space of domains with analytic boundary near to $\Omega$. We refer the reader to Theorem \ref{thm:momentflowlinear2} for a more precise version of Theorem \ref{thm:summarymoment} in this setting.   If we set $V=-\nabla p$ in the above theorem, where $p$ denotes the Green function of $\Omega,$ we recover the Hele-Shaw flow, and thus this result generalises Theorem \ref{thm:heleshawmain}.\medskip

\noindent {\bf Acknowledgements:}   We first of all wish to thank Bj\"orn Gustafsson for conversations and guidance.  We also thank Robert Berman, Bo Berndtsson, Claude LeBrun and H\r{a}kan Hedenmalm for valuable input,  and Johanna Zetterlund for assistance with creation of several images. JR also acknowledges Colin Cotter and Duncan Hewitt for patiently answering basic questions about fluid mechanics.    During this project JR was supported by an EPSRC Career Acceleration Fellowship.   \\

\noindent {\bf Terminology: } We identify $\mathbb R^2$ with the complex plane $\mathbb C$ in the standard way.  A \emph{domain} $\Omega\subset \mathbb R^2$ is a set that is open and connected and  its boundary is $\partial \Omega = \overline{\Omega} - \Omega$ where the bar denotes the topological closure.  The \emph{connectivity} of $\Omega$ is the number of connected components of $\partial \Omega$. If $\partial \Omega$ can be written locally as the graph of a smooth function then we shall say that $\Omega$ has \emph{smooth boundary} or simply that $\Omega$ is \emph{smooth}, with analogous definitions if this graph can be taken to be real analytic, and in this case there is a well defined unit outward normal vector field $n$ on $\partial \Omega$.  We say $\Omega$ is a \emph{smooth Jordan domain} if it is the interior region determined by a  smooth Jordan curve (i.e.\ a smooth non-self intersecting loop in $\mathbb R^2$). 

If $I\subset \mathbb R$ is an interval then a family of domains $\Omega_t$ for $t\in I$ is \emph{increasing} if $\Omega_t\subset \Omega_{t'}$ for $t<t'$ and the family $\Omega_t$ is said to be \emph{smooth} if the boundary $\partial \Omega_t$ can be written locally as the graph of a smooth function that depends smoothly on $t$.  If $\Omega_0$ is smooth with unit normal vector field $n$ on $\partial \Omega_0$ then an increasing smooth family $\Omega_t$ for $t\in (-\epsilon,\epsilon)$ is determined by $\partial \Omega_t = \{ x + f(x,t) n_x : x\in \partial \Omega_0\}$ for some positive smooth function $f_t(x)=f(x,t)$ on $\partial \Omega_0$, and  the \emph{normal outward velocity} of $\Omega_t$ at $t=0$ is
\[V_0:= \frac{d}{dt}\big\vert_{t=0} \partial \Omega_t := \frac{df_t}{dt}\big\vert_{t=0} n.\] 

\section{Lifting the Hele Shaw Flow}

\subsection{Complex moments} \label{moments}

Let $\Omega_t$ be a family of domains in $\mathbb R^2=\mathbb C$ and $\kappa$ be a positive smooth function.     For $k\geq 0$ the $k$th \emph{complex moment} of $\Omega_t$ is 
\[M_k(t):=\int_{\Omega_t}z^k \frac{1}{\kappa} dA,\]
 where $dA$ denotes Lebesgue measure.  Clearly then $M_0(t)$ is just the area of $\Omega_t$, and we refer to the $M_k$ for $k\ge 1$ as the \emph{higher moments}.

An important discovery by Richardson \cite{Richardson} is that these moments are conserved by the Hele-Shaw flow.   Since we are assuming the fluid is injected at a constant rate, we may as well reparametrise $t$ so that  $M_0(t) = \operatorname{Area}(\Omega_t) = M_0(0)+t$.   The upshot then is that this system is integrable, by which we mean that these moments characterise the Hele-Shaw flow:

\clearpage
 \begin{theorem}\label{thm:momentsheleshaw}\
   \begin{enumerate}
   \item  Let $\Omega_t$ be given by the Hele-Shaw flow.   Then the higher moments $M_k(t)$ for $k\ge 1$ remain constant with $t$.  

   \item  Conversely, any smooth increasing family of simply connected domains $\{\Omega_t\}$ for $t\in (-\epsilon,\epsilon)$ that contain the origin and whose higher moments $M_k$ are constant with respect to $t$ is a solution to the Hele-Shaw flow (after a possible reparameterising of $t$).
   \end{enumerate}
 \end{theorem}

The first statement is Richardson's calculation, namely that if the Hele-Shaw equation \eqref{classic3} is satisfied then

\begin{eqnarray*}
\frac{d}{dt} M_k(t) &=&   \frac{d}{dt} \int_{\Omega_t} z^k \frac{1}{\kappa} dA = \int_{\partial \Omega_t} z^k \frac{V_t}{\kappa} ds = -\int_{\partial \Omega_t} z^k \frac{\partial p_t}{\partial n} ds \\
&=& \int_{\Omega_t} \left(p_t \Delta(z^k) - z^k(\Delta p_t)\right) dA - \int_{\partial \Omega_t} p_t \frac{\partial z^k}{\partial n} ds
\end{eqnarray*}
which vanishes if $k>0$ and equals $1$ when $k=0$ as $p_t=0$ on $\partial \Omega_t$ and $\Delta(p_t) = -\delta_0$.      The second statement says in effect that this calculation can be reversed (and is, as far as we know, due to Gustafsson, \cite[p22]{Gustafsson} or \cite{Hedenmalm}).

\subsection{Schwarz functions}

We next discuss how the moment characterisation of the Hele-Shaw flow can be recast in terms of the existence of Schwarz functions.   To describe this, fix a pair $(\Omega,\phi)$ where $\Omega$ is a domain in $\mathbb R^2=\mathbb C$, and $\phi$ is a smooth function defined in some neighbourhood of $\partial \Omega$ that is strictly subharmonic (i.e.\ such that $\Delta \phi$ is positive).

\begin{definition}\label{def:schwarz}
  We say that a function $S$ is a \emph{Schwarz function} of $(\Omega,\phi)$ if $S$ is a meromorphic function on $\Omega$ which extends continuously to $\partial \Omega$ and such that
\[S(z)=\frac{\partial \phi(z)}{\partial z} \quad{} \text{ on }\partial \Omega.\] 
\end{definition}

\begin{remark}
The model case usually considered is where $\phi = |z|^2$, so $S(z) = \overline{z}$ on $\partial \Omega$.  We remark that what we define above is sometimes referred to as an interior Schwarz function, since we are requiring that it be defined on all of $\Omega$.  Much of the work of Schwarz functions asks instead that $S$ be defined, and holomorphic, on some neighbourhood of $\partial \Omega$.   It turns out, in the model case $\phi(z) = |z|^2$, that such an $S$ exists in this restricted sense if and only if $\Omega$ has real analytic boundary (see e.g.\ the book \cite{Shapiro} by Shapiro), and extends to a meromorphic function on $\Omega$ if and only if $\Omega$ is a quadrature domain \cite{Aharanov}.
\end{remark}

The connection between Schwarz functions in the model case $\phi(z) = |z|^2$ and the Hele-Shaw flow is well known (e.g.\ \cite[3.7.1]{Gustafsson}).  The following generalises this to suit our needs:

\begin{proposition}\label{prop:schwarzheleshaw}
Let $\phi$ be a smooth and strictly subharmonic function defined in a domain $D.$ Suppose $\Omega_t,$ $t\in (a,b),$ is a smooth increasing family of Jordan domains containing $0$ such that for all $t\in (a,b),$ $\partial \Omega_t\subset D,$ and for each $t$ there exists a Schwarz function $S_t$ of the pair $(\Omega_t,\phi)$ that is holomorphic on $\Omega_t$ except for having a simple pole at $0$.  Then after some reparametrization $s=s(t)$, the $\Omega_s$ are a solution to the Hele-Shaw flow with varying permeability $\kappa=1/\Delta \phi.$  
\end{proposition}
\begin{proof}

Pick $t_0\in (a,b)$ and let $t\in(t_0,b)$.  Then for $k\geq 0$ we have,
\begin{eqnarray*}
\int_{\Omega_t\setminus \Omega_{t_0}}z^k\frac{dA}{\kappa}&=&\int_{\Omega_t\setminus \Omega_{t_0}}z^k\Delta \phi dA=-\frac{i}{2}\int_{\Omega_t\setminus \Omega_{t_0}}z^kd(\frac{\partial \phi}{\partial z}dz)\\
&=&-\frac{i}{2}\int_{\partial \Omega_t}z^k\frac{\partial \phi}{\partial z}dz+\frac{i}{2}\int_{\partial \Omega_{t_0}}z^k\frac{\partial \phi}{\partial z}dz\\
&=&-\frac{i}{2}\int_{\partial \Omega_t}z^kS_tdz+\frac{i}{2}\int_{\partial \Omega_{t_0}}z^kS_{t_0}dz\\
&=&-\frac{i}{2}\int_{\Omega_t}z^k\bar{\partial}S_tdz+\frac{i}{2}\int_{\Omega_{t_0}}z^k\bar{\partial} S_{t_0}dz=\\&=&z^k(0)\pi\textrm{Res}_0(S_{t}-S_{t_0}).
\end{eqnarray*}
Thus the complex moments are preserved, and so $\Omega_t$ is a classical solution to the Hele-Shaw flow \eqref{thm:momentsheleshaw} (perhaps after reparametrization).
\end{proof}

Thus we wish to prove the existence of Schwarz functions $S_t$,  which is made easier through the freedom of choice in the subharmonic function $\phi$.  Our next aim is to construct a suitable $\phi$ that solves this problem for $t=0$.    To guide the subsequent discussion we first recall a classical argument in the case where $\phi=|z|^2.$  Consider the Cauchy integral
\[f(z):=\frac{1}{2\pi i}\int_{\partial \Omega_0}\frac{\bar{\zeta}d\zeta}{\zeta-z}.\]
 Define $f_+$ as the restriction of $f$ to $\Omega_0$ and $f_-$ as the restriction of $f$ to the complement of $\overline{\Omega}_0.$ Clearly $f_+$ and $f_-$ are both holomorphic in their respective domain of definition. A careful analysis shows that both $f_+$ and $f_-$ extend smoothly to $\partial \Omega_0$ and the Sokhotski-Plemelj formula states that on the boundary
\begin{equation} \label{sokhotski}
f_+(z)=\bar{z}+f_-(z)
\end{equation} 
(see, for example,  the book \cite{Gakhov} by Gakhov).     Moreover it can be shown that the boundary of $\Omega_0$ is real analytic if and only if both $f_+$ and $f_-$ extend holomorphically to a neighbourhood of $\partial \Omega_0$, and in this case $S:=f_+-f_-$ is holomorphic on this neighbourhood and $S(z) = \overline{z}$ on $\partial \Omega$.\medskip

We state first a simple version of what we shall prove:

\begin{proposition}\label{prop:existenceschwarz1}
  Let $\Omega_0$ be a smooth Jordan domain containing $0$.   Then there exists a smooth strictly subharmonic function $\phi$ defined on a neighbourhood $U$ of the complement of $\Omega_0$ such that $\Delta \phi =1$ on $U-\Omega_0$ and so that $(\Omega_0,\phi)$ admits a Schwarz function $S_0$ that is holomorphic except for having a simple pole at $0$.
\end{proposition}

For the problem of varying permeability we need something a little stronger:

\begin{proposition}\label{prop:existenceschwarz2}
  Let $\Omega_0$ be a smooth Jordan domain containing $0$ and $\kappa$ be a smooth positive function defined on a neighbourhood of $\partial \Omega_0$.   Then there exists a smooth function $\kappa'$ defined in on a neighbourhood $U$ of $\partial \Omega_0$ such that $\kappa'=\kappa$ on $U-\Omega_0$ and a smooth strictly subharmonic function $\phi$ defined on $U$ with $\Delta \phi = 1/\kappa'$ such that $(\Omega_0,\phi)$ admits a Schwarz function $S_0$ that is holomorphic except for having a simple pole at $0$.

Moreover if $\kappa$ is real analytic, and $\Omega_0$ has real analytic boundary then we can take $\kappa'=\kappa$. 
\end{proposition}

\begin{proof}
We need only prove Proposition \ref{prop:existenceschwarz2}.  We may as well assume $\kappa$ is defined on all of $\mathbb C$ and is smooth and positive there.   Pick some smooth function $\phi_0$ on $\mathbb C$ such that $\Delta \phi_0=1/\kappa$.

Consider the Cauchy integral \[f(z):=\frac{1}{2\pi i}\int_{\partial \Omega_0}\frac{\frac{\partial \phi_0}{\partial z}(\zeta)}{\zeta-z}d\zeta.\] As above we write $f_+$ and $f_-$ for the restriction of $f$ to $\Omega_0$ and to the complement of $\overline{\Omega}_0$ respectively.  Then the Sokhotski-Plemelj formula says that on $\partial \Omega_0,$
\begin{equation} \label{sokhotski2}
f_+(z)= \frac{\partial \phi_0}{\partial z} +f_-(z).
\end{equation} 

Observe next that $f_-$ tends to zero as $z$ tends to infinity. Thus for some constant $a$ we have that 
\[f_-(z)=\frac{a}{z}+\tilde{f}(z),\]
where $\tilde{f}$ vanishes to second order at infinity. Such a holomorphic function possesses a $\partial/\partial z$-primitive $F$ on $\mathbb C- \overline{\Omega_0}$.  Letting $h=2Re(F)$ we can rewrite equation (\ref{sokhotski2}) as 
\begin{equation}
f_+(z)-\frac{a}{z}=\frac{\partial}{\partial z}(\phi_0+h) \quad \text{ on } \partial \Omega_0.
\end{equation} 
Now set  \[\phi=\phi_0+h\quad\text{ on } \mathbb C- \overline{\Omega}_0,\]
and observe that since $h$ is the real part of a holomorphic function, $\Delta \phi = \Delta \phi_0 = 1/\kappa$.  Moreover since $f_-$ had a smooth extension to $\partial \Omega_0$ it follows that $\phi$ extends smoothly to a neighbourhood of the complement of $\Omega_0$ and we can define $\kappa'=1/(\Delta \phi)$ which will be positive near $\partial \Omega_0$.     Finally  $S_0:=f_+(z)-\frac{a}{z}$ is a Schwarz function of $(\Omega_0,\phi),$ which is holomorphic except for a simple pole at $0.$  

If $\kappa$ is real analytic and $\partial \Omega_0$ has real analytic boundary, then $f_-$ extends to a holomorphic function over a neighbourhood $U$ of $\partial \Omega_0$ \cite[1.6.4,2.5.1]{Isakov}, and $h$ similarly extends to a harmonic function over the boundary.    Thus defining $\phi$ over this neighbourhood as above we have $\Delta \phi = 1/\kappa$ over this neighbourhood as well, and so $\kappa'= \kappa$.
\end{proof}

\begin{remark}
  From the above proof one sees that in fact the second statement holds under the condition that $f_-$ extends analytically over $\partial \Omega$, the arguments below show that in this case there is a backwards solution to the Hele-Shaw flow.
\end{remark}

\begin{remark}
In what follows we will use in an essential way that $\phi$ is smooth on a neighbourhood of $\partial \Omega_0$.   But one should note that the Sokhotski-Plemelj formula holds for much more general domains (e.g. Lipschitz domains) and one still gets a Schwarz function of $(\Omega_0,\phi),$ in this case and $\phi$ will at least have a $C^{1,\alpha}$ extension.   
\end{remark}

\subsection{Holomorphic Discs}
We are now ready to phrase the Hele-Shaw problem in terms of holomorphic discs. Let $\phi$ be the subharmonic function defined on the neighbourhood $U$ of $\partial \Omega_0$ from Proposition \ref{prop:existenceschwarz2}, and $S_0$ be the Schwarz function for $(\Omega_0,\phi)$.  Define

\[\Lambda := \{ (z,\frac{\partial \phi}{\partial z}) : z\in U\} \subset \mathbb C\times \mathbb C.\]
Since this is the graph of a real strictly subharmonic function, it enjoys various properties that will be discussed in the next section.  We shall consider a meromorphic function on an open set in $\mathbb C$ as a function to the Riemann sphere $\mathbb P^1  = \mathbb C \cup \{\infty\}$ in the standard way, giving it the value $\infty$ at its poles. \medskip
\begin{definition}
Suppose $(\Omega,\phi)$ admits a Schwarz function $S$ holomorphic except for having a simple pole at $0$.  Define
\begin{equation}
\Sigma = \Sigma_{\Omega} :=\{(z,S(z))\in \mathbb C\times \mathbb P^1: z\in \Omega\}. \label{eq:definitionsigmat}
\end{equation}
\end{definition}

So $\Sigma_\Omega$ is a holomorphic disc in $\mathbb C\times \mathbb P^1$.  Moreover,
\begin{enumerate}
\item If $\pi_1$ denotes the projection to the first factor then $\pi_1\colon \Sigma\to \pi_1(\Sigma)$ is an isomorphism,
\item The boundary of $\Sigma$ lies in $\Lambda$,
\item The closure of $\Sigma$ intersects $\{0\}\times \mathbb P^1$ precisely in the point $(0,\infty)$.
\end{enumerate}

\begin{theorem}\label{thm:holomorphicdiscs}
Suppose that $\Sigma_t$ is a smooth family of holomorphic discs in $\mathbb C\times \mathbb P^1$ that satisfy the three conditions above and such that $\Omega_t:= \pi_1(\Sigma_t)$ is an increasing family of smooth Jordan domains.      Then $\Omega_t$ is a solution to the Hele-Shaw flow with permeability $\kappa' := 1/(\Delta \phi)$.  Moreover any smooth increasing family $\Omega_t$ of smooth Jordan domains that solves the Hele-Shaw flow with permeability $\kappa'$ arises in this way.

\end{theorem}
\begin{proof}
  The first statement just summarises what we have said thus far. Given $\Sigma_t,$ it can be realised as the graph of a Schwarz function $S_t$ for $\Omega_t:=\pi_1(\Sigma_t).$ Condition (3) says that $S_t$ is holomorphic except for having a simple pole at $0$.  Thus the result we want follows from Proposition \ref{prop:schwarzheleshaw}.

For the second statement, suppose that $\Omega_t$ is a smooth increasing family of smooth Jordan domains that satisfy the Hele-Shaw flow.  Let $p_t(z) = p(z,t)$ be the pressure which is required to satisfies $\Delta p_t = -\delta_0$ on $\Omega_t$ and $p_t = 0$ on $\partial \Omega_t$.  Now define

\[ u_t(z) = u(z,t) = \int_{0}^t p(z,\tau) d\tau\]
and set
\[ S_t = \frac{\partial \phi}{\partial z} - \frac{\partial u_t}{\partial z} - \chi_{\Omega_0}  \left( \frac{\partial \phi}{\partial z} - S_0\right),\]
where $\chi_{\Omega_0}$ denotes the characteristic function of $\Omega_0$.   Then $$\Delta u = \kappa^{-1}(\chi_{\Omega_t}-\chi_{\Omega_0}) - t\delta_0$$ in the sense of distributions, (see \cite[3.6]{Gustafsson} where this is considered for $\kappa\equiv 1$, or Remark \ref{rmk:weak} below) and one easily checks that $S_t$ is a Schwarz function for $\Omega_t$ which is holomorphic except for a simple pole at $0$ .  Thus the $\Omega_t$ are induced from the corresponding family $\Sigma_{\Omega_t}$ as required.
\end{proof}

In the following section we show that the moduli space of such holomorphic discs has real dimension 1, and thus our original disc $\Sigma_0$ deforms in a smooth one dimensional family $\Sigma_t$.    Observe that condition (1) and the property that $\pi_1(\Sigma_t)$ is a smooth Jordan domain are open conditions, and thus since they hold for $\Sigma_0$ they will hold for $\Sigma_t$ for small $t$ as well.

\begin{remark}
  In much the same way one can interpret the Hele-Shaw flow with multiple inputs, in which case condition (3) above must be suitably modified.  We leave the details to the reader.
\end{remark}

\begin{remark}\label{rmk:weak}
 To understand the long-time behaviour of the Hele-Shaw flow, and also for more general starting domains, one has a weak formulation of the problem. Indeed, following Elliott-Janovsky \cite{Elliott}, Gustafsson \cite{Gustafsson2} and Sakai \cite{Sakai}, \cite{Sakai4} it is possible to cast a weak formulation of the Hele-Shaw flow as an obstacle problem.

For simplicity assume here that $\kappa\equiv 1$.   Let $f$ be a distribution on some domain $D.$ A function $u$ on $D$ is said to solve the corresponding obstacle problem if
\begin{equation} \label{obstacle}
-\Delta u-f\geq 0, \qquad{} u\geq 0, \qquad{} (-\Delta u-f)u=0.
\end{equation} 
The free boundary $\Gamma(u)$ of the obstacle problem is the boundary of the coincidence $\Lambda(u):=\{x: u(x)=0\}.$ If the domain $D$ has a boundary one also specifies a non-negative function $g$ on $\partial D$ and demand that the solution $u$ should, in addition to satisfying (\ref{obstacle}), be equal to $g$ on the boundary.

Now  let $D=\mathbb{R}^2$ and for any $t>0$ set 
\[f_t:=\chi_{\Omega_0}-1+t\delta_0.\] By general potential theory  the corresponding obstacle problem has an unique lower semicontinuous solution $u_t$ and one calls the family \[\Omega_t:=\{z: u_t(z)>0\}\] the weak solution to the Hele-Shaw flow with initial domain $\Omega_0.$   It is know that a smooth solution, if it exists, will coincide with the weak solution, and if the weak solution is regular enough, it is the smooth solution (see e.g. \cite{Gustafsson}).  Moreover there does exists a weak solution that is defined for all time, but in general the $\Omega_t$ will eventually be non-smooth or cease to be simply connected.

Now suppose  $S_0$ is a Schwarz function of $(\Omega_0,\phi)$ and suppose that $\Delta \phi=1$ on the complement of $\Omega_0.$ We have seen that such a function exists at least when $\Omega_0$ is  Lipschitz. Let $\Omega_t$ be the weak solution to the Hele-Shaw flow. We observe that for any $t\in[0,\infty),$ there exists a Schwarz function for $(\Omega_t,\phi),$ holomorphic except for a simple pole at $0.$ To see this, let $u_t$ be solution to the above obstacle problem described. Then one simply notes that the function 
\[S_t:=(\frac{\partial \phi}{\partial z}-\frac{\partial u_t}{\partial z})-\chi_{\Omega_0}(\frac{\partial \phi}{\partial z}-S_0)\]
 is the claimed Schwarz function.

Hence using the same lifting we thus see that there is a natural deformation of our initial holomorphic disc $\Sigma_0$ that exists for \emph{all} positive time, but whose members need not be smooth and may have more exotic topology.  However the possibilities are constrained, as there are strong regularity results due to Sakai \cite{Sakai2,Sakai3} for the boundaries $\partial \Omega_t$ of the domains $\Omega_t$ belonging to a weak Hele-Shaw flow which says that if $\overline{\Omega}_0\subseteq \Omega_t$ (which, for example is  the case if $\partial \Omega_0$ is $C^1)$ then $\partial \Omega_t$ is real analytic except for having at most a finite number of cusps or double points.   
\end{remark}

\section[Empty initial condition]{Short time existence of the Hele Shaw flow with empty initial condition}

We are interested in smooth solutions to the Hele-Shaw flow with varying permeability $\kappa$ and empty initial condition.  That is, we wish to find a smooth solution $\Omega_t$ such that $\text{Area}(\Omega_t)=t$ for $t\in (0,\epsilon)$ for some $\epsilon>0$.     Observe that this problem is homogeneous in $\kappa$, by which we mean that the solution is unchanged if $\kappa$ is replaced by $c\kappa$ for some $c\in \mathbb R_{>0}$ (after a reparameterisation of the time variable).  This is clear from the defining equations of the Hele-Shaw flow, and can also be seen from the characterization in terms of the complex moments $M_k(t)$ which all scaled by a factor of $c^{-1}$ when $\kappa$ is replaced by $c\kappa$.

Our proof is based on a connection between the Hele-Shaw flow and a certain complex Homogeneous Monge-Amp\`ere Equation (HMAE).  We recall that our result extends a theorem of Hedenmalm-Shimorin \cite{Hedenmalm} who prove short-time existence under the assumption $\kappa$ is real analytic and $\log \kappa$ is strictly subharmonic.  

\subsection{The Homogeneous Monge-Amp\`ere Equation}

The complex Homogeneous Monge-Amp\`ere Equation (HMAE) is a higher dimensional analogue of the Laplace equation on complex manifolds.    We first describe some work of Donaldson \cite{Donaldson} that gives an equivalence between regular solutions of a certain HMAE and families of holomorphic discs with boundaries in a so called ``LS submanifold'' which will be defined below.

Let $X$ be a compact K\"ahler manifold of complex dimension $n,$ and let $\omega_0$ be a K\"ahler form. This means that locally $\omega_0$ is equal to $i\partial \bar{\partial}\phi$ where $\phi$ is a smooth strictly plurisubharmonic function. For example, if $X$ has an ample line bundle $L,$ and $h$ is a positive metric on $L,$ then the curvature form $i\Theta_h$ is a K\"ahler form on $X$ which lies in the first Chern class of $L.$ Any K\"ahler class in the same cohomology class as $\omega_0$ can be written as $\omega_0+i\partial\bar{\partial}\phi$ where $\phi$ is a global smooth function.

Let $D$ denote the closed unit disc in $\mathbb C$ and $D^{\times} = D-\{0\}$.  Also let $\Omega_0$ be the pullback of $\omega_0$ by the projection $\pi\colon X\times D\to D$.  Suppose next that $F$ is a smooth function on $X\times \partial D$.  We shall write $F_{\tau}(\cdot) = F(\cdot,\tau)$ thought of as a function on $X$, and suppose that $\omega_0+i\partial \bar{\partial}F_{\tau}$ is a K\"ahler form for all $\tau \in \partial D$.   A \emph{regular solution} to the associated HMAE is an extension of $F$ to a smooth function $\Phi$ on $X\times D$ such that $\omega_0+i\partial \bar{\partial}\Phi_{\tau}$ is a K\"ahler form for all $\tau\in D$ and 
\[(\Omega_0+i\partial \bar{\partial}\Phi)^{n+1}=0\] on $X\times D.$  
We observe that clearly finding a solution to this boundary value problem is unchanged if $F$ is replaced by $F+c$ for some constant $c$.  

Now suppose $\Phi$ is a regular solution to the above HMAE. Then the kernel of $\Omega_0+i\partial \bar{\partial}\Phi$ defines a one dimensional complex distribution, which one can show is integrable, thus giving a foliation of $X\times D$ by holomorphic discs transverse to the $X$ fibres (see \cite{Bedford} and \cite{Donaldson}). \medskip

A further connection with holomorphic discs comes about through the following idea due to Donaldson \cite{Donaldson} and independently Semmes \cite{Semmes}.  Let $\Theta:=\Theta_1+i\Theta_2$ denote a holomorphic $2$-form on a complex manifold $W,$ where $\Theta_1$ and $\Theta_2$ are real symplectic forms. A (real) submanifold $Y$ of $W$ is called a \emph{$LS$ submanifold} if it is Lagrangian with respect to $\Theta_1$ while it is symplectic with respect to $\Theta_2.$   

The point of this definition is that there exists a holomorphic fibre bundle $W_X$ over $X$ such that K\"ahler forms in the same class as $\omega_0$ correspond to closed LS submanifolds in $W_X.$  We sketch its construction: if $\omega_0$ has a potential $\phi_U$ in some open set $U,$ we identify $W_U$  with the $(1,0)$ part of the complex cotangent bundle. If $z_i$ are local holomorphic coordinates any $(1,0)$-form can be written as $\sum_i \xi_i dz_i,$ thus $(z_i,\xi_j)$ are local holomorphic coordinates of $W_U.$ We also have non-degenerate holomorphic $2$-form, namely $\Theta:=\sum_i d\xi_i\wedge dz_i.$ 

If $V$ is some other open set where $\omega_0$ has a potential $\phi_V,$ then on $U\cap V$ the transition function is given by $\partial(\phi_V-\phi_U).$ It means that the section, locally defined as $\partial \phi_U,$ is in fact globally defined. By a simple calculation the graph of this section turns out to be a LS submanifold in $W_X.$ Let $\omega$ be some other K\"ahler form in the same class as $\omega_0.$ Then $\omega=\omega_0+i\partial \bar{\partial}\psi$ for some smooth global function $\psi,$ and the section $\partial\phi_U+\partial \psi$ defines a section of $W_X.$ By the same reasons this is also a LS submanifold. Donaldson shows that K\"ahler forms in the class of $\omega_0$ are in one-to-one correspondence with the closed LS submanifolds of $W_X.$    

Now let $F$ be a smooth function on $X\times \partial D$ such that $\omega_0+i\partial \bar{\partial}F_{\theta}$ is positive for all $\theta\in \partial D$.    For each $\theta\in \partial D,$ let $\Lambda_{\theta}$ denote the LS submanifold in $W_X$ corresponding to $\omega_0+ i\partial\bar{\partial} F_{\theta}$.

\begin{theorem}[Donaldson  {\cite[Theorem 3]{Donaldson}}]\label{thm:donaldson1}
There is a regular solution to the complex HMAE with boundary values given by $F$ iff there is a smooth family of holomorphic discs $g_x: D \to W_X$ parametrised by $x\in X$ with
\begin{enumerate}
\item $\pi(g_x(0))=x \in X;$
\item for each $\theta \in \partial D$ and each $x\in X,$ 
\[g_x(\theta)\in \Lambda_{\theta};\]
\item for each $\tau \in D$ the map $x \mapsto \pi(g_x(\tau))$ is a diffeomorphism from $X$ to $X.$ 
\end{enumerate}
In fact for fixed $\tau\in D,$ the image of the map $x \mapsto g_x(\tau)$ is the LS submanifold $\Lambda_{\tau}$ associated to the K\"ahler form $\omega_0 + i\partial \bar{\partial} \Phi_{\tau}$.
\end{theorem}  

Using this theorem Donaldson applies the theory of the moduli spaces for embedded holomorphic discs with boundaries in a totally real submanifold to deduce the following:

\begin{theorem}[Donaldson {\cite[Theorem 1]{Donaldson}}]
  The set of boundary conditions $F$ for which the HMAE above has a regular solution $\Phi$ is open in the $C^2$ topology. 
\end{theorem}

\subsection{Short-time existence}

We now return to the short time existence problem with empty initial condition.   Fix a potential $\phi$ in the unit ball in $\mathbb C$ with $\Delta \phi = 1/\kappa$, where $\kappa$ is the positive smooth function that encodes the permeability for the Hele-Shaw flow.  

In the usual way think of $\mathbb C$ embedded in $\mathbb{P}^1=\mathbb{C}\cup \{\infty\}$.  We let $\omega_{FS}$ be the standard Fubini-Study metric on $\mathbb P^1$, which has potential $\phi_{FS} = \log(1+|z|^2)$ on $\mathbb C$.    Using a partition of unity we can, without loss of generality,  assume that $\phi$ is defined on all of $\mathbb C$ and so that $\phi-\phi_{FS}$ extends to a smooth function at infinity (so said another way, $\phi$ is the local potential of some globally defined potential whose curvature is in the same K\"ahler class as $\omega_{FS}$).  

We also consider the $\mathbb{C}^*$-action on $\mathbb{C}^2$ given by
\[\rho(\lambda): (z,w) \mapsto (\lambda z, \lambda^{-1} w), \quad\lambda \in \mathbb C^*.\]
which extends to an action on $\mathbb P^1\times \mathbb C$.  We also denote by $\rho$ the induced action of $S^1$ on $\mathbb P^1\times S^1$.

\begin{proposition}\label{prop:hmae}
There exists an $c\in \mathbb R_{>0}$ and a smooth function $\tilde{F}$ on $\mathbb P^1\times S^1$ such that 
\begin{enumerate}
\item There exists a regular solution to the HMAE with boundary data $\omega_{FS} + i\partial\bar{\partial} \tilde{F}_{\theta}$ for $\theta\in S^1$.  \item There is are constants $c,r\in \mathbb R_{>0}$ such that $\omega_{FS} + i \partial \bar{\partial} \tilde{F}_{\theta}$ has local potential $c\phi(\theta z)$ for $(z,\theta)\in B_r\times S^1$ where $B_r$ is the ball of radius $r$ centered at the origin.
\item The function $\tilde{F}$ is invariant under the $S^1$-action $\rho$.
\end{enumerate}
\end{proposition}

Assuming this for now we prove the short term existence result.  Let $\tilde{\Lambda}_{\theta}$ denote the LS-manifold associated to $\omega_{FS} + i\partial\bar{\partial}\tilde{F}_{\theta}$.    Under the natural identification of $W_{\mathbb{C}}$ with $\mathbb{C}^2\to \mathbb C$ the second condition tells us that
 \[\tilde{\Lambda}_{\theta} \cap (B_{\delta}\times \mathbb C)=\{ (z, c \frac{\partial}{\partial z} \phi(\theta z)) : z\in  B_{\delta}\} = \{(z,c\theta\frac{\partial \phi}{\partial z}(\theta z)): z\in B_{\delta}\}.\]
Moreover as $\tilde{F}$ is invariant under $\rho$, the automorphism $\rho(\theta)$ for $\theta\in S^1$  maps the LS manifold $\tilde{\Lambda}_{\theta}$ to $\tilde{\Lambda}_1.$    \medskip

 By  Theorem \ref{thm:donaldson1} there is an associated family of holomorphic discs $g_x$ in $W_{\mathbb{P}^1}.$  Since the function $\tilde{F}$ is invariant under the $S^1$-action $\rho,$ the solution $\Phi$ is also invariant, and in particular the restriction $\Phi_0$ to the central fibre is $S^1$-invariant. Set $\phi_0 := \Phi_0 + \log(1+|z|^2)$.   Then $\phi_0$ is an $S^1$-invariant strictly subharmonic function, which implies that $\frac{\partial \phi_0}{\partial z}(0)=0$ and $\frac{\partial \phi_0}{\partial z}(z)\neq 0$ whenever $z\neq 0.$ This means that $g_0(0)=(0,0),$ while for any $x\neq 0,$ the second component of $g_x(0)$ is not $0.$  

Furthermore the family of discs $g_x$ is invariant under $\rho$, i.e.
\[\rho(\theta)g_x(\theta \tau) = g_{\theta x}(\tau) \quad \text{ for } \theta\in \partial D,  \tau \in D.\]
In particular 
\begin{equation} \label{invariance}
g_{0}(\theta)=\rho(\theta^{-1})(g_0)(1) \quad \text{ for } \theta \in \partial D
\end{equation}
and since $\pi(g_0(0))=0$ we see $\pi g_0$ must be the constant map to $0$.  So, by continuity of the family $g$, we have that as long as $|x|$ is sufficiently small, the  image of $g_x: D \to W_{\mathbb{P}^1},$ $x\neq 0,$ lies in $W_{\mathbb{C}}\cong \mathbb{C}^2$.

So fix such an $x$ with $x\neq 0$.   Then by twisting the disc using the $\mathbb{C}^*$-action $\rho$ we get a new punctured disc $f_x: D^{\times} \to \mathbb{C}\times \mathbb{P}^1$ given by 
\[f_x(\tau):=\rho(\tau)(g_x(\tau)).\] 
Since the second coordinate of $g_x(0)$ is not zero  $f_x$ extends to a map $D:\mathbb{C}\times \mathbb{P}^1$.  Moreover for $\theta \in \partial D$ and $\tau\in D$  we have 
\[f_{\theta x}(\tau) = \rho(\tau) g_{\theta x}(\tau) = \rho(\tau \theta) g_x(\theta \tau)= f_x(\theta\tau)\]
so the holomorphic disc defined by $f_x$ depends only on $t=|x|$, which we shall now denote by  $\Sigma_t$.    Observe that  $f_x(0)=(0,\infty)$ so $\Sigma_t$ meets $\{0\}\times \mathbb P^1$ at $(0,\infty)$.  Moreover  if $\tau\in \partial D$ then $g_x(\tau)\in \tilde{\Lambda}_{\tau}$.   Thus as long as $t$ is sufficiently small, the boundary of $\Sigma_t$ lies in 
$$\Lambda := \{ (z,c\frac{\partial \phi}{\partial z}) \}.$$

Recall that the map $\gamma(x):=\pi g_x(1)$ is a diffeomorphism of $\mathbb{P}^1$ to itself. Since, by invariance $\pi f_x(\theta)=\pi g_{\theta x}(1)=\gamma(\theta x)$ for $\theta \in \partial D$ we deduce
$$\partial \Omega_t = \{ \pi f_x(\theta) : \theta \in \partial D\} = \{\gamma(\theta x): \theta \in \partial D\}$$
which is precisely the image of the circle of radius $t=|x|$ under $\gamma$.    In particular this shows that, for small $t>0$, $\Omega_t$ is an increasing family of smooth Jordan domains containing the origin, and $\pi\colon \Sigma_t\to \pi(\Sigma_t)$ is an isomorphism and the limit $\Omega_t$ as $t$ tends to zero is the origin.    Hence by Theorem \ref{thm:holomorphicdiscs} the family $\Omega_t  = \pi(\Sigma_t)$ satisfies the Hele-Shaw flow with permeability $\kappa:=1/(c\Delta \phi)$.  Thus by the homogeneity of the Hele-Shaw problem with empty initial condition, this solves the Hele-Shaw flow with permeability $\kappa$ as well after a reparamterisation of time.\medskip

Thus it remains to prove Proposition \ref{prop:hmae}.    Observe that if $\phi$ happened to be rotation invariant, i.e. $\phi(\tau z)=\phi(z)$ for all $\tau\in \partial D,$ then the function
\begin{equation}
  \label{eq:F}  
F(z,\theta):=\phi(\theta z)-\phi_{FS}(z).
\end{equation}
would also be invariant and the HMAE has a trivial regular solution, namely $\Phi(z,\tau):=\phi(z)-\log(1+|z|^2)$  (this reflects the elementary observation that if the permeability $\kappa$ is radially symmetric then the Hele-Shaw flow with empty initial condition will be the trivial solution in which $\Omega_t$ is a disc centered at the origin.)   Of course it this need not be the case, but we will show that any potential $\phi$ is close (in a precise sense) to one that is invariant, at which point we can apply Donaldson's openness theorem.

We will need following regularisation of the maximum of two smooth functions, which we prove for completeness.

 \begin{lemma}\label{lem:regmax}
Let $\alpha= \alpha(z)$ and $\beta = \beta(z)$ be real valued smooth functions on the closure of a bounded open subset $U$ of $\mathbb C$.  Let $f$ be a smooth non-negative bump function on the real line with unit integral and support in $(-a,a)$ for some $a>0$ and define $u\colon U\to \mathbb C$ by
   \begin{eqnarray*}
u &=& \int_{\mathbb R} \max\{ \alpha, \beta + \lambda\} f(\lambda) d\lambda.
   \end{eqnarray*}
   \begin{enumerate}
   \item If $z$ is such that $\alpha(z) > \beta(z) + a$ then $u(z) = \alpha(z)$.  Similarly if $\alpha(z)<\beta(z)-a$ then $u(z) = \beta(z)$.
   \item The function $u$ is smooth on $U$ and
$$ \| u -\alpha\|_{C^2} \le  a + \| \alpha-\beta\|_{C^2} + \|d\alpha-d\beta\|_{C^0}^2 \|f\|_{C^0}$$
\end{enumerate}
\end{lemma}

\begin{proof}
If $\alpha(z)>\beta(z) +a$ then $\sup\{\alpha,\beta+\lambda\} = \alpha$ for all $\lambda\in (-a,a)$ and so $u(z) = \int_{\lambda} \alpha(z) f(\lambda) d\lambda = \alpha(z)$.  The case when $\alpha(z)<\beta(z)-a$ is similar.

Now setting $\gamma = \alpha-\beta$ we have
\[u =     \int_{-\infty}^{\gamma} \alpha f(\lambda) d\lambda + \int_{\gamma}^{\infty} (\beta + \lambda) f(\lambda) d\lambda\]
which shows that $u$ is smooth and
$$ u-\alpha = \int_{\gamma}^{\infty} (\beta-\alpha)f(\lambda) d\lambda  + \int_{\gamma}^{\infty} \lambda f(\lambda)d\lambda$$
since $\int_{\lambda} \alpha f(\lambda)d\lambda =\alpha$.  Since $\lambda f(\lambda) = 0$ if $|\lambda|>a$ this implies
$$ \|u-\alpha\| \le \|\alpha-\beta\|_{C^0} + a.$$

Now for any any smooth functions $\alpha,\beta,\gamma$ on $U'$, we have from the Chain rule
$$d\int_{-\infty}^{\gamma} \alpha f(\lambda) d\lambda = \int_{-\infty}^{\gamma} d\alpha f(\lambda) d\lambda + \alpha (f\circ \gamma) d\gamma$$
and similarly
$$d\int_{\gamma}^{\infty} (\beta+\lambda) f(\lambda) d\lambda = -\int_{\gamma}^{\infty} d\beta f(\lambda)d\lambda - (\beta+\gamma)(f\circ\gamma) d\gamma.$$
Thus putting $\gamma = \alpha-\beta$ gives
 \begin{eqnarray}\label{eq:chain}
 du&=& \int_{-\infty}^{\gamma} (d\alpha) f(\lambda) d\lambda + \int_{\gamma}^{\infty} (d\beta) f(\lambda) d\lambda
 \end{eqnarray}
and so
 \begin{eqnarray}\label{eq:chain}
 du-d\alpha&=&  \int_{\gamma}^{\infty} (d\beta-d\alpha) f(\lambda) d\lambda
 \end{eqnarray}
yielding $|du-d\alpha| \le \|\alpha-\beta\|_{C^1}$.

For the second derivative a further differentiation of \eqref{eq:chain} gives that when $D = \frac{\partial ^2}{\partial x\partial y}$,
$$ Du -D\alpha= \int_{\gamma}^{\infty} (D\beta-D\alpha) f(\lambda) d\lambda + (f\circ \gamma) \frac{\partial(\beta-\alpha)}{\partial y} \frac{\partial \gamma}{\partial x}$$
and similarly for the other second order derivatives.  Thus
\begin{eqnarray*}
    |Du-D\alpha| &\le&  \| \alpha-\beta\|_{C^2} + \|d\alpha-d\beta\|_{C^0}^2 \|f\|_{C^0}
  \end{eqnarray*}
 which completes the proof.
\end{proof} 

Returning to the Hele-Shaw flow, the initial data consists of a real-valued smooth strictly subharmonic $\phi$ which we may write locally as
\[\phi(z)=\textrm{Re}(\alpha_0+\alpha_1z+\alpha{2}\bar{z}+\alpha_3 z^2+\alpha_4 \bar{z}^2)+\beta |z|^2+o(|z|^2),\] where $\beta$ is a real positive number. Our requirement on $\phi$ was that $\Delta\phi = 1/\kappa$,  so without loss of generality we can assume that this first term is zero since it is harmonic.  By rescaling there is no loss of generality if we assume $\beta=1$ and so in fact
$$ \phi(z) = |z|^2 + o(|z|^2).$$

Now let $\psi$ be an $S^1$-invariant smooth strictly subharmonic function on $\mathbb C$ that is equal to $|z|^2$ on the unit disc, and such that $\psi-\log(1+|z|^2)$ extends to a smooth function as $|z|$ tends to infinity.

\begin{lemma}\label{lem:max}
For all $\epsilon>0$ there exists a smooth strictly subharmonic $\phi^{\epsilon}$ and $\delta_1,\delta_2, r>0$ such that
\begin{enumerate}
\item $\phi^{\epsilon} = (1-\delta_1) \phi - \delta_2$ on the ball $B_r = \{ |z|<r\}$
\item $\phi^{\epsilon} = \psi$ for $|z|>1$
\item $\|\phi^{\epsilon} - \psi\|_{C^2}<\epsilon.$
\end{enumerate}
\end{lemma}
\begin{proof}
  Let $\epsilon>0$ be small.   Write $\phi(z) = |z|^2+ g(z)$ where $g(z) = O(|z|^3)$, so say $|g(z)|<C|z|^2$ on the unit ball.  Fix a small $R\in (0,1/5)$ so that $\|g\|_{C^2}<\epsilon/20$ on $B_R$.    By shrinking $R$ if necessary we may assume also that $CR<\epsilon/20$.    We observe for later use the $C^2$-bound actually implies $|g'(z)|\le \epsilon R/20$ on $B_R$.  

Now let $\delta_1:=\epsilon/20$ and $\delta_2 := \epsilon R^2<\epsilon/20$ and $a:=\delta_2/2$.   Fix a smooth non-negative bump function $f$ with unit integral and support in $[-a,a]$ and $\|f\|_{C^2}\le 2a^{-1}$.  Putting $\alpha = \psi$ and 
$$ \beta = (1-\delta_1)\phi - \delta_2$$
into Lemma \ref{lem:regmax} consider, the function on  $B_R$ given by
$$\phi^{\epsilon} = \int_{\lambda} \max\{ \psi, (1-\delta_1)\phi - \delta_2+ \lambda\} f(\lambda) d\lambda$$
is smooth.  Moreover it is strictly subharmonic since both $\psi$ and $\phi$ are.

If $|z|=R$ then $(1-\delta_1)|g(z)|< 2CR^3 \le \epsilon R^2/20 = \delta_1 R^2$, and so
$$\alpha(z) = \psi(z) = R^2 > (1-\delta_1) (R^2 - g(z)) - \delta_2/2 = \beta(z)+a\quad \text{ on } \partial B_R$$
Thus $\alpha>\beta$ on a neighbourhood of $\partial B_R$, and so  by Lemma \ref{lem:regmax} $\phi^{\epsilon}  = \alpha= \psi$ on this neighbourhood.   Hence by declaring $\phi^{\epsilon}$ to be equal to $\psi$ outside of $B_R$ we get that $\phi^{\epsilon}$ is a smooth strictly subharmonic function on all of $\mathbb C$, and (2) certainly holds.  We next notice that $(1-\delta_1)\phi(0) - \delta_2 = -\delta_2=-2a= \psi(0) - 2a$.  Thus there exists a small ball $B_r$ around $a$ for which $\beta = (1-\delta_1)\phi - \delta_2 < \psi -a$, and so $\phi^{\epsilon} = \beta$ on $B_r$ by Lemma \ref{lem:regmax} which gives (1).  

It remains to bound the $C^2$ norm of $\phi^{\epsilon} - \psi$.   Observe that there is nothing to show outside of $B_R$.  On the other hand on $B_R$ we have
$$ \alpha-\beta = \psi - ((1-\delta_1) \phi - \delta_2) = \delta_1 |z|^2 - (1-\delta_1) g(z) + \delta_2.$$
Thus the $C^2$-norm of $\alpha-\beta$ is bounded by $4\delta_1 + 2\epsilon/20 + \delta_2<\epsilon/2$.  Moreover since $|g'(z)|\le \epsilon R/20$ on $B_R$ we have
$$ |d\alpha -d\beta| \le \delta_1 3R + 2\epsilon R/20 \le \epsilon R/4$$

Thus as $\|f\|_{C^0}\le 2a^{-1} = 2/\delta_2$, Lemma \ref{lem:regmax} gives

$$ \|\alpha-\beta\|_{C^2} \le \epsilon/2 + (\epsilon R/4)^2 \|f\|_0 \le \frac{\epsilon}{2}  + \frac{\epsilon^2 R^2}{8 \delta_2} = \epsilon/2  + \frac{\epsilon^2 R^2}{8 \epsilon R^2} < \epsilon$$
which gives (3).
\end{proof}

\begin{proof}[Proof of Proposition \ref{prop:hmae}]
Let $\psi$ be as above, so is an $S^1$-invariant potential.  Then the function $$F_0(z,\theta) := \psi(z) - \phi_{FS}(z)$$ is independent
  of $\theta$, and so there certainly exists a regular solution to the HMAE
  with boundary data given by 
$$ \omega_{\tau} : = \omega_{FS} + i\partial \bar{\partial} F_0 = i\partial \bar{\partial} \psi$$
namely the trivial one obtained by pulling $i\partial \bar{\partial} \psi$ to the product.
Thus by Donaldson's Openness Theorem there exists an $\epsilon>0$ and a regular solution to the HMAE with boundary data determined by any smooth function  $G$ with $\|G-F_0\|_{C^2}<\epsilon$.

Now let $\phi^{\epsilon}, \delta_1,\delta_2,r$ be as in Lemma \ref{lem:max} and set 
$$\tilde{F}(z,\theta) := \phi^{\epsilon}(\theta z) - \phi_{FS}(z).$$

Then by construction $\tilde{F}$ is invariant under the action of $\rho$ and by the Lemma $\|\tilde{F} - F_0\|_{C^2} = \|\phi^{\epsilon} - \psi\|_{C^2}<\epsilon.$  Hence there is a solution to the HMAE with boundary data $\omega_{FS} + i\partial \partial \tilde{F}_{\tau}$.  Moreover by the properties of $\phi^{\epsilon}$ we know there is a $B_r$ so that $\phi^{\epsilon} = (1-\delta_1) \phi + \delta_2$ on $B_r$.  So letting $c= (1-\delta_1$, this says precisely that $\omega_{FS} + i \partial \bar{\partial} \tilde{F}$ has local potential $c\phi(\theta z)$ for $(z,\theta)\in B_r\times S^1$
\end{proof}

\section[Holomorphic curves]{Holomorphic curves with boundaries in a totally real submanifold} 

To continue to exploit our connection between the Hele-Shaw flow and the theory of holomorphic discs we now summarize the parts of the literature that we shall need.

\subsection{The moduli space}
Let $M$ be a complex manifold of complex dimension $n.$ A real submanifold $\Lambda$ is said to be \emph{totally real} if $T\Lambda\cap J(T\Lambda)=\{0\},$ where $J$ denotes the complex structure on $M$.   A totally real submanifold is said to be \emph{maximal} if it has real dimension $n$. 

Let $(\Sigma,\partial \Sigma)$ be a Riemann surface with boundary.  A holomorphic non-const\-ant map $f\colon \Sigma \to M$ is called a \emph{holomorphic curve} in $M$ if $f$ extends to a $C^1$ function on $\partial \Sigma$.  If $f$ is an embedding then we call the image of $f$ an \emph{embedded holomorphic curve}.  By abuse of terminology we shall refer to this image as $\Sigma$, and say that it has boundary in $\Lambda$ if $f(\partial \Sigma)\subset \Lambda$.  In the special case that $\Sigma$ is the unit disc $D\subset \mathbb C$ and $f(\partial D)\subset \Lambda$ we call the image an \emph{embedded holomorphic disc} with boundary in $\Lambda$. By the smooth Riemann mapping theorem \cite[Thm 3.4]{Cho}, one can equivalently replace the disc $D$ with any smooth Jordan domain.

The moduli space of holomorphic curves with boundary in a totally real maximal submanifold has been studied by numerous authors in different contexts, e.g. Bishop \cite{Bishop}, Donaldson \cite{Donaldson}, Forstneri\v c \cite{Forstneric} and Gromov \cite{Gromov}. In complex analysis this is related to polynomial hulls and the homogeneous Monge-Ampère equation \cite{Donaldson}, while in symplectic topology it is used in the study of Lagrangian submanifolds \cite{Gromov}. More recently LeBrun and Mason \cite{LeBrun2, LeBrun} have shown a connection to twistor theory.  \medskip

The first step in this theory is to give the space of all (not necessarily holomorphic) embedded curves in some regularity class with boundary in $\Lambda  $ the structure of a Banach manifold $\mathcal{B}$.   Let $\Sigma$ be such an embedded curve, let $N$ denote the normal bundle and let $E$ be the normal bundle of $\partial \Sigma$ in $\Lambda$. We can identify $E$ with a totally real subbundle of $N$ restricted to $\partial \Sigma.$ Locally around $\Sigma$ the Banach manifold $\mathcal{B}$ is modelled on the Banach space $C^{k+1,\alpha}(\Sigma, N,E)$ of sections of $N$ with boundary values in $E$ (this construction uses the exponential map with respect to some appropriately chosen metric, see \cite{LeBrun}).

The next step is to consider a Banach vector bundle $\pi: \mathcal{E}\to \mathcal{B}$ whose fibre over $\Sigma$ is \[C^{k,\alpha}(\Sigma,\Lambda^{0,1}\otimes N)\] and a canonical section $\mathcal{S}:\mathcal{B}\to \mathcal{E}$ whose zero locus $\mathcal M:=\mathcal{S}^{-1}(0)$ is the space of holomorphic curves. The implicit function theorem for Banach spaces says that $\mathcal{M}$ is a smooth manifold provided that the graph of 
$\mathcal{S}$ is transverse to the zero section. This is so if and only if at any point $\Sigma$ in $\mathcal{M},$ the linearisation $D_{\Sigma},$ which by definition is $D\mathcal S$ composed with the projection to the fibre of $\mathcal{E},$ is surjective and has a right inverse.  Now $D_{\Sigma}$ is nothing but the Cauchy-Riemann operator $\bar{\partial}$ that maps sections of $N$ with boundary values in $E$ to $(0,1)$-forms with values in $N.$ It is well-known that this operator is Fredholm, so what remains to establish transversality of $\mathcal{S}$ is surjectivity. Thus if the Cauchy-Riemann operator is surjective, then $\mathcal{M}$ is a smooth manifold and its dimension is given by the index of $\bar{\partial}.$ 

In general the problem of proving transversality involves picking a new generic almost complex structure on the manifold $M$.  However for our purposes $M$ has complex dimension 2 and then the problem becomes significantly easier and, as we shall see, reduces to the computation of a single topological invariant of a related holomorphic curve.

\subsection{Dimension count and the Schottky double}

Following LeBrun in \cite{LeBrun}, we now discuss how to use the Schottky double to compute the dimension of the moduli space $\mathcal M$ of embedded holomorphic curves in a manifold $M$ of complex dimension $2$.

Let $\Sigma$ be a compact Riemann surface with boundary $\partial \Sigma.$ By gluing it together with its complex conjugate $\bar{\Sigma}$ along $\partial \Sigma$ we get a closed Riemann surface $\mathcal{X}$ (assuming some regularity of $\Sigma$, see \cite{LeBrun}). In the theory of quadrature domains $\bar{\Sigma}$ is called the Schottky double of  $\Sigma$ \cite{Gustafsson4}. The complex conjugate $\bar{N}$ of $N$ is a holomorphic vector bundle on $\bar{\Sigma},$ and by gluing $N$ and $\bar{N}$ along $\partial \Sigma$ so that $E=\bar{E}$ we get a holomorphic vector bundle $\mathcal{N}$ on $\mathcal{X}.$

\begin{figure}[htb]
\includegraphics[width=0.8\textwidth, trim=200 200 200 200]{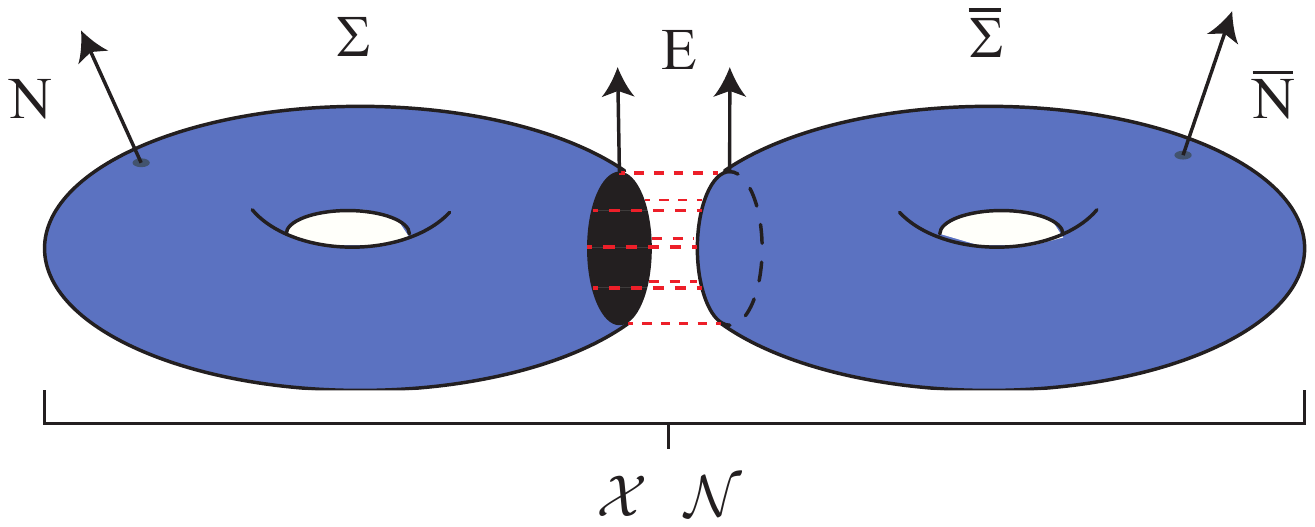}  
\centering
\caption{The Schottky double}
\end{figure}
There is an antiholomorphic involution $\rho$ on $\mathcal{X}$ which lifts to $\mathcal{N}$ such that $E$ is the fixed point set of $\rho.$   As noted in \cite{LeBrun}, this implies that $E$ is real analytic as a subspace of $\mathcal{N}.$   This involution thus acts on each of the vector spaces
\[C^{k+1,\alpha}(\Lambda,\mathcal{O}(\mathcal{N}))\text{ and } H^p(\mathcal{X},\mathcal{O}(\mathcal{N})) \text{ for } p=0,1.\]
Let $C^{k+1,\alpha}_{\rho}(\mathcal{X},\mathcal{O}(\mathcal{N}))$ and $H^p_{\rho}(\mathcal{X},\mathcal{O}(\mathcal{N}))$ denote the invariant subspaces. Thus
\[H^p(\mathcal{X},\mathcal{O}(\mathcal{N}))=H^p_{\rho}(\mathcal{X},\mathcal{O}(\mathcal{N}))\oplus iH^p_{\rho}(\mathcal{X},\mathcal{O}(\mathcal{N})).\]

LeBrun shows in \cite{LeBrun} that $$C^{k+1,\alpha}(\Sigma, N, E) \cong C^{k+1,\alpha}_{\rho}(\mathcal{X},\mathcal{O}(\mathcal{N}))$$ with a canonical isomorphism. Also \cite[Lemma 1]{LeBrun} the kernel of the Cauchy-Riemann operator is canonically isomorphic to $H^0_{\rho}(\mathcal{X},\mathcal{O}(\mathcal{N}))$ while the cokernel is canonically isomorphic to $H^1_{\rho}(\mathcal{X},\mathcal{O}(\mathcal{N})).$ Thus $\bar{\partial}$ is surjective if and only if  $h^1(\mathcal{X},\mathcal{O}(\mathcal{N}))=0,$ and then the index is equal to $h^0(\mathcal{X},\mathcal{O}(\mathcal{N})).$

Now suppose that the complex dimension of the ambient manifold $M$ is two.  Then $\mathcal{N}$ is a complex line bundle, and so as long as $\deg \mathcal N>2g-2$, where $g$ is the genus of $\mathcal{X}$, then the numbers $h^0(\mathcal{X},\mathcal{O}(\mathcal{N}))$ and $h^1(\mathcal{X},\mathcal{O}(\mathcal{N}))$ are topological invariants that only depend on $g$ and the degree of $\mathcal{N}$.

If $\Lambda$ happens to be the fixed point set of an antiholomorphic involution $\rho$ of $M,$ then the Schottky double $\mathcal{X}$ of $\Sigma$ can be identified with the Riemann surface $\Sigma\cup \rho(\Sigma),$ and $\mathcal{N}$ is then identified with the normal bundle of  $\Sigma\cup \rho(\Sigma)$ in $M$.  Hence a strategy to compute the numbers $h^0(\mathcal{X},\mathcal{O}(\mathcal{N}))$ and  $h^1(\mathcal{X},\mathcal{O}(\mathcal{N}))$ for a general pair $(\Sigma,\Lambda)$ in the four dimensional case is to continuously perturb $(\Sigma,\Lambda)$ to some $(\Sigma',\Lambda')$ where $\Lambda'$ is the fixed point set of an involution $\rho,$ and then use the adjunction formulae to compute the degree of the normal bundle of $\Sigma'\cup \rho(\Sigma').$

\section{Moduli spaces of quadrature domains}\label{sec:quadrature}

We digress in this section from the Hele-Shaw flow, and consider the above setup in the case $\Lambda=\{(z,\bar{z})\}$ which is related to the concept of quadrature domains.

\begin{definition}
A bounded domain $\Omega$ in $\mathbb{C}$ is called a \emph{quadrature domain} if there exist finitely many points $a_1,...,a_m\in \Omega$ and coefficients $c_{kj}$ such that for any integrable holomorphic function $f$ on $\Omega$ the quadrature identity holds: 
\[\int_{\Omega}f dA=\sum_{k=1}^{m}\sum_{j=0}^{n_k-1}c_{kj}f^{(j)}(a_k).\]
 The integer $n=\sum_{k=1}^m n_k$ is called the \emph{order} of the quadrature domain (assuming $c_{k,n_k-1}\neq 0$).    We say that a quadrature domain is smooth if its boundary is smooth.   We also recall that the \emph{connectivity} $c$ of a domain $\Omega$ is the number of connected components of $\partial \Omega$.
\end{definition}

Assume that $(\Omega,\phi(z) = |z|^2)$ has a Schwarz function $S$.   Then 
\begin{eqnarray*}\int_{\Omega}fdA=-\frac{i}{2}\int_{\Omega}fd(\bar{z}dz)&=&-\frac{i}{2}\int_{\partial \Omega}f\bar{z}dz\\&=&-\frac{i}{2}\int_{\partial \Omega}fSdz=-\frac{i}{2}\int_{\Omega}f\bar{\partial}Sdz\\&=&\sum_{k=1}^{m}\sum_{j=0}^{n_k-1}c_{kj}f^{(j)}(a_k),
\end{eqnarray*}
where $a_k$ are the poles of $S$ and $n_k$ is the order of the pole at $a_k$, and thus $\Omega$ is a quadrature domain.   It is proved in \cite[Lemma 2.3]{Aharanov} that the converse holds as well.   Clearly then the order of the domain $n$ equals the number of poles of $S$ counted with multiplicity. 

It is also clear that Schwarz function $S$ is uniquely determined by the domain, since if we had two Schwarz functions, their difference would be a meromorphic function that was zero on the boundary, and therefore this difference must be zero.\medskip

From now on assume that the quadrature domain $\Omega$ has smooth boundary. Consider the graph \[\Sigma:=\{(z,S(z)): z\in \Omega\}\] as an embedded holomorphic curve in $\mathbb C\times \mathbb{P}^1.$ It is attached along its boundary to the totally real submanifold 
\[\Lambda=\{(z,\bar{z}): z\in \mathbb{C}\}.\] Since $\Lambda$ is the fixed point set of the antiholomorphic involution $\rho: (z,w)\mapsto (\bar{w},\bar{z}),$ the Schottky double $\mathcal{X}$ is naturally identified with $\Sigma\cup \rho(\Sigma)$ and the normal bundle $\mathcal{N}$ is identified with the normal bundle of  $\Sigma\cup \rho(\Sigma).$  

We wish to calculate the degree of $\mathcal{N}.$ By the adjunction formula,
\[\mathcal{N}=\mathcal O(\mathcal{X})_{|\mathcal{X}}=(K_{\mathbb{P}^1\times \mathbb{P}^1})^*_{|\mathcal{X}}\otimes K_{\mathcal{X}}\]
where $\mathcal O(\mathcal X)$ denotes the line bundle associated to the divisor $\mathcal X$.  The line bundle $K_{\mathbb{P}^1\times \mathbb{P}^1}$ is represented by the divisor $-2(H_1+H_2)$ where $H_1$ and $H_2$ denote the hyperplanes $\mathbb{P}^1\times\{\infty\}$ and $\{\infty\}\times \mathbb{P}^1$ respectively. Thus 
\[\textrm{deg}\,\mathcal{N}=2(H_1+H_2)\cdot \mathcal{X}+\textrm{deg}\,K_{\mathcal{X}}.\]
where the dot represents the intersection number of the relevant divisors.   Clearly 
\[H_1\cdot \mathcal{X}=H_2\cdot \mathcal{X}=n,\] and \[\textrm{deg}\,K_{\mathcal{X}}=2g-2,\] where $g$ denotes the genus of $\mathcal{X}.$ Combined this yields that \[\textrm{deg}\,\mathcal{N}=4n+2g-2.\]  Now $n>0$, so  $\textrm{deg}\,\mathcal{N}\geq 2g+2$ which implies $h^1(\mathcal{X},\mathcal{N})=0$, and  by Riemann-Roch that \[h^0(\mathcal{X},\mathcal{N})=\textrm{deg}\,\mathcal{N}-g+1=4n+g-1.\] Finally note that the genus $g$ is precisely $c-1,$ where $c$ is the connectivity number of $\Omega.$ 

Hence the moduli space of embedded holomorphic curves with boundaries in $\Lambda$ is a smooth manifold of real dimension $4n+c-2$ near $\Sigma.$ Any such holomorphic curve lying close enough to $\Sigma$ gives a Schwarz function of a smoothly bounded quadrature domain. By the fact that all quadrature domains have unique Schwarz functions (again with respect to $\phi(z) = |z|^2$), it follows that the moduli space of smoothly bounded quadrature domains of order $n>0$ and connectivity $g$ is a manifold of real dimension $4n+c-2.$ We have thus proved Theorem \ref{quadraturetheorem} as stated in the introduction.

\section[Short-time existence]{Short-time existence of the Hele-Shaw flow}\label{sec:proofshorttime}

This section completes the proof of Theorem \ref{thm:heleshawmain}.   To ease notation let  $\Omega=\Omega_0$ be a smooth Jordan domain. We saw above that we could construct a generalised Schwarz function $S$ of the pair $(\Omega,\phi)$ where $\Delta \phi=1/\kappa$ on the complement of $\Omega$ and $\phi$ extends smoothly to a neighbourhood of the complement. If $\Sigma$ denotes the graph of $S$ and $\Lambda$ the graph of $\frac{\partial \phi}{\partial z}$ we have that $\Sigma$ is a holomorphic disc whose boundary lies in the totally real submanifold $\Lambda$.

We want to analyse the moduli space of holomorphic discs near $\Sigma$ with boundary in $\Lambda$ that intersects the hyperplane $H_1=\mathbb{P}^1\times \{\infty\}$ at the point $(0,\infty).$ 

First we consider the moduli space of all such discs, irrespective of their intersection with $H_1$.  We need to calculate the numbers $h^0(\mathcal{X},\mathcal{N})$ and $h^1(\mathcal{X},\mathcal{N}).$ As we continuously deform the domain $\Omega$ to the unit disc $\Omega'$, the pair $(\mathcal{X},\mathcal{N})$ and submanifold $\Lambda$ continuously deform to the pair $(\mathcal{X}',\mathcal{N}')$ and to $\Lambda'=\{(z,\bar{z}): z\in \mathbb{C}\}$ respectively. Thus the Schwarz function for $\Omega'$ is simply $\frac{1}{z},$ i.e. the usual (interior) Schwarz function of the unit disc thought of as a quadrature domain. Since the numbers $h^0(\mathcal{X},\mathcal{N})$ and $h^1(\mathcal{X},\mathcal{N})$ are invariant under continuous deformation, the calculation in the previous section yields that $h^1(\mathcal{X},\mathcal{N})=0$ and $h^0(\mathcal{X},\mathcal{N})=3,$ since $n=1$ and $g=0.$ Thus the required moduli space is a smooth manifold of real dimension 3 near $\Sigma$.

Considering moduli spaces of holomorphic curves passing through a given point is also standard in the theory (see, for example, \cite{McDuff}). We simply replace our Banach manifold $\mathcal B$ with the subspace  of curves going through the point, which is locally modelled on the Banach subspace of sections to the normal bundle that vanish at that point. We will denote this space by $C^{k+1,\alpha}(\Sigma, N, E)_0$. Following \cite{LeBrun}  it follows that this space is naturally isomorphic to the space of $\rho$ invariant sections of $\mathcal{N}$ that vanish at both $(0,\infty)$ and $\rho((0,\infty)).$ We denote this space by $C^{k+1,\alpha}_{\rho}(\mathcal{X},\mathcal{N})_0.$ The section $\mathcal{S}$ and its linearisation $\bar{\partial}$ are simply the restrictions of the original ones. 

We check first that $\bar{\partial}$ is still surjective. Pick an arbitrary element $$g\in C^{k,\alpha}(\Sigma,\Lambda^{0,1}\otimes N).$$ By the above there exists a section $f\in C^{k+1,\alpha}(\Sigma,N,E)$ such that $\bar{\partial} f=g.$ Since the degree of $\mathcal{N}$ is positive and $\mathcal{X}$ is isomorphic to $\mathbb{P}^1$ we know that there is a holomorphic section, $s$ say,  of $\mathcal{N}$ that is equal to $f$ at $(0,\infty)$ but zero at $\rho((0,\infty)).$  Then the section $s+\rho(s)$ lies in $H^0_{\rho}(\mathcal{X},N),$ and thus corresponds to an element $\tilde{s}$ in the kernel of $\bar{\partial}.$ We have that $\tilde{s}=f$ at the point $(0,\infty).$ Thus $f-\tilde{s}$ lies in $C^{k+1,\alpha}(\Sigma,N,E)_0$ and $\bar{\partial}(f-\tilde{s})=g.$ This shows that the restriction of $\bar{\partial}$ is still surjective.      

The kernel of the restricted Cauchy-Riemann operator is of course the intersection of the original kernel with the subspace $C^{k+1,\alpha}(\Sigma, N, E)_0.$ Under the natural isomorphism this equals the intersection of $H^0_{\rho}(\mathcal{X},N)$ and $C^{k+1,\alpha}_{\rho}(\mathcal{X},\mathcal{N})_0.$ Let $H^0(\mathcal{X},N)_0$ denote the subspace of holomorphic sections that vanish at both $(0,\infty)$ and $\rho((0,\infty)).$ The involution $\rho$ acts on this space by \[\rho(s)(x):=\rho(\bar{s(\rho(x))}),\] and if $H^0_{\rho}(\mathcal{X},N)_0$ denotes the invariant real subspace we have that 
\begin{equation} \label{eqlast}
H^0(\mathcal{X},N)_0\equiv H^0_{\rho}(\mathcal{X},N)_0\oplus iH^0_{\rho}(\mathcal{X},N)_0.
\end{equation} 
Clearly $H^0(\mathcal{X},N)_0$ is the intersection of $H^0_{\rho}(\mathcal{X},N)$ and $C^{k+1,\alpha}_{\rho}(\mathcal{X},\mathcal{N})_0,$ and from (\ref{eqlast}) we get that the real dimension of $H^0_{\rho}(\mathcal{X},N)_0$ is equal to the complex dimension of $H^0(\mathcal{X},N)_0.$ Since the degree of $\mathcal{N}$ is two and $\mathcal{X}\cong \mathbb{P}^1$ the complex dimension of $H^0(\mathcal{X},N)_0$ is one. Thus the index of the restricted Cauchy-Riemann operator is one.

Together with the surjectivity this implies that the moduli space of nearby holomorphic discs with boundary in $\Lambda$ and going through $(0,\infty)$ is a smooth manifold of real dimension one.   Thus $\Sigma_0=\Sigma$ varies in a smooth family $\Sigma_t$ for $t$ sufficiently small, and by Theorem \ref{thm:holomorphicdiscs} it remains to prove that the family of domains  $\Omega_t = \pi(\Sigma_t)$ is increasing.

The tangent space of the family of discs at $\Sigma$ is given by $H^0_{\rho}(\mathcal{X},N)_0$ restricted to $\Sigma.$ Since a non-zero element in $H^0_{\rho}(\mathcal{X},N)_0$ vanishes only at the points $(0,\infty)$ and $\rho((0,\infty))$ it is nonvanishing along the boundary of $\Sigma.$ The real subspace $E$ is one-dimensional and is has one direction going out of $\Omega_0$ and one direction going in. The nonvanishing implies that a nonzero tangent points either outwards everywhere or inwards everywhere. Thus (replacing $t$ with $-t$ if necessary) the family $\Omega_t$ is increasing.

\begin{remark}
One sees that in fact we have proved slightly more than claimed, namely that exists a smooth positive function $\kappa'$ defined in a neighbourhood of $\partial \Omega_0$ with $\kappa=\kappa'$ outside of $\Omega_0,$ and an $\epsilon >0$ such that there exists a smooth increasing family of domains $\{\Omega_t\}_{t\in (-\epsilon,\epsilon)}$ which solves the Hele-Shaw flow with varying permeability $\kappa'.$    This extension $\kappa'$ is not unique, but a given choice uniquely determines the flow, in the sense that any two such flows agree on their common domain.  

In fact if $\kappa\equiv 1$ and the initial domain is simply connected (which we always assume is the case) then it is known that the boundary of the domain being real analytic is necessary for there to be a solution backwards in time (see \cite{Tian2} or \cite[Thm 10]{Gustafsson2}).   Thus we see that this issue can be circumvented if we modify the permeability appropriately (although we have little control on this modification).
\end{remark}

\section{The inverse potential problem and moment flows}

Let $\rho$ be a density function on $\mathbb{C}$ (so we have changed notation from previous sections in which this density was played by $1/\kappa$).  The associated potential of a domain $\Omega$ is the subharmonic function $$U^{\Omega}(z):=\int_{\Omega} \log|z-w|^2\rho(w)dA(w)$$
where $dA$ is the Lebesgue measure.  Note that $U^{\Omega}$ is harmonic on the complement of $\Omega.$ Let $D$ be some big disc containing $\Omega$. The inverse potential problem asks whether a given harmonic function $U$ on the complement of $D$ (with proper growth at infinity) can be realised (uniquely) as the potential of a domain $\Omega$ in $D.$ \medskip

To see the relationship with the complex moments consider
\[G:=\frac{\partial U}{\partial z},\] 
which is holomorphic if $U$ is harmonic.  If $U=U^{\Omega}$ for some domain $\Omega$ then 
\begin{equation} \label{potential}
G(z)=\frac{\partial}{\partial z} \int_{\Omega} \log|z-w|^2\rho(w)dA(w)=\int_{\Omega}\frac{\rho(w)}{z-w}dA(w).
\end{equation} 
Now for large $|z|$ we can write this integral as a convergent power series in $z^{-k},$ $$\int_{\Omega}\frac{\rho(w)}{z-w}dA(w)=\sum_{k=1}^{\infty}a_k z^{-k},$$ where $$a_k:=\int_{\Omega}w^{k-1}\rho(w)dA(w).$$ Thus the coefficients $a_k$ of the power series of $G$ at infinity equals the complex moments of $\Omega$ with respect to the measure $\rho dA.$      Thus the inverse potential problem be restated as asking for a (unique) domain with prescribed complex moments.\medskip

This potential problem has a rich history, both for the plane and for higher dimensions, and also for more complicated kernels replacing the expression $\ln |z-w|^2$ above (see, for example, \cite{Isakov} and the references therein).  For this particular problem in the plane  it is known that prescribing the complex moments does not uniquely determine the domain (although it appears to be unknown if this uniqueness can fail for domains with smooth boundary).  However starting with a suitable smooth domain $\Omega_0$, for small variations $U_t$ of the potential $U^{\Omega_0}$ there is a unique $\Omega_t$ close to $\Omega$ such that $U_t = U^{\Omega_t}$.  Thus one can think of the set of moments as giving local parameters for the space of domains.  We refer the reader to \cite[Sec 2.4]{Isakov} and \cite[Sec. 2]{Cherednichenko} for precise statements.  We also point out the interesting work of Wiegmann and Zabrodin in this context, which connects these parameters to an integrable system coming from the dispersionless Toda lattice \cite{Wiegmann}.  \medskip

We have seen that the Hele-Shaw flow can be phrased as a local existence question related to the inverse potential problem.    Namely, if $\Omega$ is a smooth Jordan domain containing the origin, then there is a non-trivial one parameter family of smooth domains $\Omega_t,$ all having the same higher moments as $\Omega$ and whose area increases linearly.     In this section we will prove short-time existence of a similar flow in which all the higher moments are allowed to vary linearly.    Variants of the Hele-Shaw flow under some other force (e.g.\ gravity) are considered by Escher-Simonett \cite{Escher2, Escher3,Escher} and this gives a physical interpretation to some of these variations.  

\begin{theorem}\label{thm:momentflowlinear}
Let $\Omega$ be a smooth Jordan domain, and $V$ be a smooth, nowhere vanishing outward pointing normal vector field on $\partial \Omega$.   Then there exists a smooth variation $\Omega_t$ for $t\in [0,\epsilon)$ for some $\epsilon>0$ such that $$\frac{d}{dt}\big\vert_{t=0}\partial \Omega_t=V$$ whose moments vary linearly in $t$.
\end{theorem}

Note that without the assumption on the vector field $V$ pointing outwards this theorem is not true, for example in the Hele-Shaw flow when $\rho\equiv 1$ and $V=\nabla p,$ then this would mean solving the Hele-Shaw flow backwards in time, which is impossible unless the boundary of $\Omega$ is real analytic.\medskip

Of course the way in which the moments vary will depend on $V$.  To discuss this suppose $V$ is any normal vector field defined on the boundary of a smooth Jordan domain $\Omega.$   Write the measure $V\rho ds$ as the restriction of a complex one form $g(z)dz$ to the curve $\partial \Omega.$ The complex-valued function $g$ is then uniquely determined on $\partial \Omega.$ Consider the Cauchy integral $$f=f^{\Omega,V}=\int_{\partial \Omega}\frac{g(w)}{z-w}dw,$$ and as above let $f_+=f_+^{\Omega,V}$ denote the restriction of $f$ to $\Omega$ and $f_-=f_-^{\Omega,V}$ the restriction to the complement of $\overline{\Omega}.$ Recall that both $f_+$ and $f_-$ has a smooth extension to the boundary and \[f_+=g+f_- \quad \text{ on } \partial \Omega.\]

Now if $\Omega_t$ is a smooth variation of $\Omega$ with initial direction $V$, the variation of the complex moments $M_k(t)$ of $\Omega_t$ satisfy
$$\frac{dM_k}{dt}|_{t=0}=\int_{\partial \Omega}z^kV\rho ds.$$ 
Since $f_+$ is holomorphic in $\Omega$ it then follows that $$\frac{dM_k}{dt}|_{t=0}=-\int_{\partial \Omega}z^kf_-(z)dz.$$ Thus writing $f_-$ as a power series in a neighbourhood of infinity, $$f_-=\sum_{k=1}^{\infty}b_k z^{-k}$$ we conclude $$\frac{dM_k}{dt}|_{t=0}=-2\pi i b_{k+1}.$$  
We think of the condition that $f_-$ extend to an analytic function across $\partial \Omega$ as a kind of ``growth condition'' on the coefficients $b_k$ that determine the variation of the moments.  With this said we can prove something more precise:

\begin{theorem}\label{thm:momentflowlinear2}
Let $\Omega$ be a smooth Jordan domain.    Suppose complex coefficients $b_k$ for $k\ge 0$ are given with $ib_0\in \mathbb R$, so that the function $f_- = \sum_k b_k z^{-k}$ defined in a neighbourhood of infinity extends to a holomorphic function across $\partial \Omega$.  Then there exists a smooth density function $\rho'$ with $\rho'= \rho$ outside of $\Omega$ and a smooth family $\{\Omega_t\}_{t\in (-\epsilon,\epsilon)}$, whose moments $M_k(t)$ taken with respect to the measure $\rho' dA$ vary linearly as
\[ M_k(t) = M_k(0) - 2\pi i b_{k+1} t.\]
If moreover $\Omega$ has real analytic boundary and $\rho$ is real analytic then we can take $\rho'=\rho$.
\end{theorem}

\begin{remark}
This theorem gives one precise way to state that these moments form parameters for the space of domains near $\Omega$.  The final statement, under the extra assumption that the boundary is real analytic, appears to be well known and to have been proved in a number of contexts; it follows, for instance, from \cite[Sec. 2]{Cherednichenko}.  We are not aware of a similar statement in the smooth case, or one such as Theorem \ref{thm:momentflowlinear} that uses the extra data that the initial normal vector field is outward pointing.
\end{remark}

The proof starts with the following uniqueness result:

\begin{proposition} \label{uniquenessvectorfields}
 Let $V_1$ and $V_2$ be two smooth normal vector fields defined on the boundary of $\Omega.$ If for all $k\geq 0$ we have that 
\begin{equation} \label{gtgtgt}
\int_{\partial \Omega}z^kV_1\rho ds= \int_{\partial \Omega}z^kV_2\rho ds,
\end{equation}
then it follows that $V_1=V_2.$
\end{proposition}

\begin{proof}
Let $g_1$ and $g_2$ be the complex valued functions so that $V_i\rho ds=g_idz$ on $\partial \Omega$ for $i=1,2.$ By our calculations above it follows from (\ref{gtgtgt}) that $g_1-g_2$ extends holomorphically to $\Omega.$ Let $u$ be a conformal map from the unit disc to $\Omega$ and consider the pullback of $(g_1-g_2)dz$ to the unit disc. It will be a holomorphic one form which can be written as $h(w)dw,$ where $w$ denotes the holomorphic coordinate on the unit disc. The fact that $(g_1-g_2)dz$ restricts to a real measure on $\partial \Omega$ implies that the restriction of $h(w)dw$ to the unit circle is real. Thus in turn implies that $wh(w)$ is real valued on the unit circle. It follows that $wh(w)\equiv 0$ and thus $g_1=g_2,$ and consequently $V_1=V_2.$ 
\end{proof}

Recall that we have a totally real submanifold $\Lambda$ which is the graph of $\frac{\partial \phi}{\partial z}$ where $\Delta\phi=\rho$ outside of $\Omega.$ We also have a holomorphic disc $\Sigma$ attached to $\Lambda$ which projects down to $\Omega.$ As we saw in Section \ref{sec:proofshorttime}, $\Sigma$ belongs to a one parameter family of holomorphic discs attached to $\Lambda$ which goes through the point $(0,\infty),$ and the projection of these yields the Hele-Shaw flow (the permeability being $1/\rho$).

We record the following theorem.

\begin{theorem} \label{perturbationtheorem}
Let $\Sigma$ and $\Lambda$ be as above. Given any small smooth perturbation $\Lambda_s$ of $\Lambda,$ the family $\Sigma_t$ deforms to a smooth family $\Sigma_{s,t}$ of holomorphic discs going through $(0,\infty)$ attached to $\Lambda_s$ along their boundaries. \end{theorem}

\begin{proof}
This is standard in the theory of moduli spaces of embedded holomorphic curves (see e.g.\ \cite{LeBrun}). Let $\Lambda_s$ be a smooth parametrised perturbation of $\Lambda,$ $s\in(-\epsilon,\epsilon).$ The Banach space $\mathcal{B}$ one considers is the space of embedded discs in some regularity class together with the parameter $s,$ where the discs passes through $(0,\infty)$ and is attached to $\Lambda_s.$ The  Banach bundle $\mathcal{E}$ is defined just as before, and so is the section $\mathcal{S}.$ The linearisation of $S$ is the $\bar{\partial}$ operator which maps the tangent space of $\mathcal{B}$ to the fibre of $\mathcal{E}.$ Let $\Sigma_s$ be a smooth family of smooth discs lying in $\mathcal{B}$ such that $\Sigma_0=\Sigma$ and $\Sigma_s$ attaches to $\Lambda_s.$ Then the tangent space of $\mathcal{B}$ is spanned by $C^{k+1,\alpha}(\Sigma, N, E)_0$ together with the derivative of $\Sigma_s$ at $s=0,$ which we will denote by $f.$ We have already established that the linearisation $\bar{\partial}$ is surjective, even when restricted to the subspace $C^{k+1,\alpha}(\Sigma, N, E)_0$. Therefore there is a section $g$ in $C^{k+1,\alpha}(\Sigma, N, E)_0$ such that $\bar{\partial}f=\bar{\partial}g.$ Thus the kernel of $\bar{\partial}$ is spanned by the kernel of the restricted $\bar{\partial}$ operator together with $f-g,$ so it has dimension $2.$ By the implicit function theorem for Banach manifolds the space of holomorphic discs nearby $\Sigma$ lying in $\mathcal{B}$ is a manifold of real dimension $2.$ Also the subsets which are attached to $\Lambda_s$ for a given $s$ will be one dimensional submanifolds.
\end{proof}

We now prove Theorem \ref{thm:momentflowlinear}. Let $f_-$ be the restriction to the complement of $\bar{\Omega}$ of 
\[ f= \int_{\partial \Omega} \frac{g(w)}{z-w} dw,\]
where $V\rho ds = g(z)dz$, and choose some smooth extension to a neighbourhood $U$ of the boundary of $\Omega.$   We will consider the smooth perturbation of $\Lambda$ given by $$\Lambda_s:=\{(z,\frac{\partial \phi}{\partial z}+sf_-): z\in U\supset \partial \Omega\}.$$

By Theorem \ref{perturbationtheorem} the family of holomorphic discs corresponding to the Hele-Shaw flow $\Sigma_t$ perturbs to a smooth family $\Sigma_{s,t}$ with boundary in $\Lambda_s$.   For each $s$ select the holomorphic disc $\Sigma'_s:=\Sigma_{s,t(s)}$ so that the corresponding Schwarz function $S_s$ has the same residue as $S_0.$ That $\Sigma'_s$ is a smooth family follows from the ordinary implicit function theorem in finite dimensions. Let $\Omega_s$ be the projection of $\Sigma'_s$, we have for small $s$,
\begin{eqnarray}
\int_{\Omega_s}z^k\rho dA-\int_{\Omega_0}z^k\rho dA =\int_{\partial \Omega_s}z^k\frac{\partial \phi}{\partial z}dz -\int_{\partial \Omega_0}z^k\frac{\partial \phi}{\partial z}dz=  \nonumber \\  =\int_{\partial \Omega_s}z^k(S_s-sf_-)dz-\int_{\partial \Omega_0}z^kS_0dz=-s\int_{\partial \Omega_s}z^kf_-dz. \label{momentcalculation}
\end{eqnarray} 
Here we used the fact that $S_s=\frac{\partial \phi}{\partial z}+sf_-$ on $\partial \Omega_s$ and that $S_s$ and $S_0$ have the same residue at $0.$ 

If $V_0$ is the initial normal vector field on $\partial \Omega$ corresponding to the perturbation $\Omega_s$, then \eqref{momentcalculation} implies that $V_0$ and $V$ induce the same initial variation of moments. From Proposition \ref{uniquenessvectorfields} it thus follows that the initial direction of $\Omega_s$ is given by $V.$   In particular this initial direction is, by assumption, non-vanishing and outward pointing, so by smoothness $\partial \Omega_s$ will lie in the complement of $\bar{\Omega}$ for small $s.$   But on  this complement $f_-$ is holomorphic, and $$\int_{\partial \Omega_s}z^kf_-dz= 2\pi i b_{k+1}$$ (which in particular is independent of $s$).  Thus by \eqref{momentcalculation} again, all the moments of $\Omega_s$ vary linearly as required.

The proof of Theorem \ref{thm:momentflowlinear2} is identical, only now since we do not have the hypothesis that the flow is outward initially we must take the moments with respect to $\rho' = \Delta \phi$ which is only equal to $\rho$ outside of $\Omega$ unless real analytic is assumed (as in Proposition \ref{prop:existenceschwarz2})

\noindent {\sc Julius Ross,  University of Cambridge, UK. \\j.ross@dpmms.cam.ac.uk}\vspace{2mm}\\ 
\noindent{\sc David Witt Nystrom,  University of Gothenburg, Sweden. \\\quad wittnyst@chalmers.se}

\end{document}